\documentclass[12pt, a4paper, reqno]{amsart}
\usepackage[utf8]{inputenc}

\usepackage{color}
\usepackage[margin=1.1in,top=1.4in,bottom=1.4in]{geometry}
\usepackage{amsmath, amsthm, amssymb, amsfonts}
\usepackage{calc}
\usepackage{url}
\usepackage{enumitem}
\usepackage{booktabs}

\usepackage{cleveref}
\crefformat{section}{\S#2#1#3} 
\crefformat{subsection}{\S#2#1#3}
\crefformat{subsubsection}{\S#2#1#3}

\numberwithin{equation}{section}
\newtheorem{theorem}{Theorem}[section]

\newtheorem{lemma}[theorem]{Lemma}

\theoremstyle{definition}

\theoremstyle{remark}

\title{Asymptotics for smooth numbers in short intervals}

\author{Khalid Younis}
\address{Mathematics Institute, Zeeman Building, University of Warwick, Coventry CV4 7AL,
United Kingdom}
\email{khalid.younis@warwick.ac.uk}

\begin{document}

\begin{abstract}A number is said to be $y$-smooth if all of its prime factors are less than or equal to $y.$  For all $17/30<\theta\leq 1,$ we show that the density of $y$-smooth numbers in the short interval $[x,x+x^{\theta}]$ is asymptotically equal to the density of $y$-smooth numbers in the long interval $[1,x],$ for all $y \geq \exp((\log x)^{2/3+\varepsilon}).$ Assuming the Riemann Hypothesis, we also prove that for all $1/2<\theta\leq 1$ there exists a large constant $K$ such that the expected asymptotic result holds for $y\geq (\log x)^{K}.$

 Our approach is to count smooth numbers using a Perron integral, shift this to a particular contour left of the saddle point, and employ a zero-density estimate of the Riemann zeta function.
\end{abstract}
\maketitle
\section{Introduction}
A positive integer is \(y\)\emph{-smooth} (or \(y\)-\emph{friable}) if all of its prime factors are less than or equal to \(y.\)   Of central importance to understanding the distribution of smooth numbers is to first count (at least approximately) how many there are in a large interval \([1,x].\) This question has received significant attention over the last century, and has profited from both combinatorial and analytic approaches (see the surveys \cite{hilten2} and \cite{granv}).
 
 To further understand their distribution, it is natural to ask how many smooth numbers there are in an interval \((x,x+h],\) where \(h\) is large but much smaller than \(x.\) 
If the smooth numbers were distributed not too erratically, then, for \(h\) and \(y\) in a suitable range depending on \(x\), one might expect  the density of smooth numbers in short intervals and in long intervals to be the same. That is, as \(x\rightarrow \infty,\)  
\begin{equation}\label{asym}\frac{|\{n\in(x,x+h]: n \textrm{ is } y\textrm{-smooth}\}|}{h}\sim \frac{|\{n\in [1,x]:  n \textrm{ is } y\textrm{-smooth}\}|}{x}. \end{equation}  

By first counting numbers with a single very large prime factor and then employing a `recursive identity' for the set of smooth numbers, the first result establishing a relation of the form  \eqref{asym}  was  by Hildebrand \cite{hild2} on the range 
\[x/y^{5/12}\leq h\leq x \quad \textrm{ and } \quad
e^{(\log \log x)^{5/3+\varepsilon}} \leq y \leq x.\] The restriction on \(h\) stems from the well-known estimates of Huxley on the prime number theorem in short intervals, and the \(y\) bound comes from an application of the best known error terms in the prime number theorem.  It has been suggested, for instance in \cite{friedgran}, \cite[\S 5]{hilten2}, and \cite[\S 4]{granv}, that ideas of Granville \cite{granv2} may enable the \(5/12\) to be improved to \(1-\varepsilon.\)

Hildebrand and Tenenbaum \cite{hilten} later proved an asymptotic of the shape  \eqref{asym} (though with the right-hand side  multiplied by a positive quantity \(\alpha=\alpha(x,y)\)  known as the \emph{saddle point}, which we define later) valid in  a wide range of \(y\) (which is slightly more tricky to state) and for which it is necessary, but not sufficient, to have
\[x/ e^{(\log x)^{3/5}}\leq h\leq x.\]
Friedlander and Granville \cite{friedgran} later showed, by exploiting an `almost-all' result for smooth numbers, the asymptotic \eqref{asym}  holds  in the range 
\[\sqrt{x}y^2e^{(\log x)^{1/6}}\leq h\leq x \quad \textrm{ and } \quad
e^{(\log x)^{5/6+\varepsilon}} \leq y \leq x.\]

Let us briefly discuss the current best results concerning lower bounds for the number of \(y\)-smooth numbers in \((x,x+h].\) A more thorough discussion is presented in the surveys \cite{hilten2} and \cite{granv}, as well as in the following works. 
 The breakthrough work of Matom\"{a}ki  and Radziwi\l\l   ~ on  multiplicative functions in short intervals \cite{matrad} showed that there are at least \(\sqrt{x}/(\log x)^{4}\) many \(y\)-smooth numbers for \(y=x^{\varepsilon}\) and \(h=C(\varepsilon)\sqrt{x}.\) Matom\"{a}ki \cite{matom} showed that for
 \(y=e^{(\log x)^{2/3}(\log \log x)^{4/3+\varepsilon}}\) and \(h=x^{1/2+o(1)}\) a lower bound of magnitude \(x^{1/2-\varepsilon}\) is attainable. Earlier work of Xuan \cite{xuan} demonstrated the existence of such smooth numbers in essentially this regime. Conditional on the Riemann Hypothesis, Soundararajan  \cite{sound2} proved that one may take \(y=e^{5\sqrt{\log x \log \log x}}\) and \(h=x^{1/2+o(1)},\) also for a lower bound of magnitude \(x^{1/2-\varepsilon}.\)  These works explicitly describe their \(o(1)\) terms. 
 
 Returning to asymptotic results, suppose we wanted the short interval to be rather small, say  \(h= x^{\theta}\) for some fixed \(0<\theta<1.\)  We see that the  results of Hildebrand, and of Hildebrand and Tenenbaum, stated earlier offer nothing in this scenario. On the other hand, the Friedlander and Granville result can handle this (with \(\theta\) as low as \(1/2+o(1)\))  provided that \( e^{(\log x)^{5/6+\varepsilon}}\leq y\leq x^{1/4-\varepsilon}.\) 
In this article, we improve on this with regards to the smoothness by showing that we may take \(y\) as low as \(e^{(\log x)^{2/3+\varepsilon}}\) and still have asymptotic \eqref{asym} hold for any \(17/30<\theta\leq 1.\)  Here and throughout, we write \(\Psi(x,y)=|\{n\in[1,x]: n \textrm{ is } y\textrm{-smooth}\}|\) and \(u=(\log x)/\log y.\)
\begin{theorem}\label{main}
Let \(17/30<\theta\leq 1.\) There exists a constant \(C=C(\theta)>0\)  such that the estimate \[\frac{\Psi(x+h,y)-\Psi(x,y)}{h}=\frac{ \Psi(x,y) }{x}\left(1+O_\theta\left(\frac{\log (u+1)}{\log y}\right)\right)\] holds uniformly for \begin{equation}\label{yandh} x^{\theta}\leq h\leq x \quad \textrm{ and } \quad
e^{C(\log x)^{2/3}(\log\log x)^{4/3}} \leq y \leq 2x.\end{equation}
\end{theorem}
The analogous result for primes, that is, the number of primes in the interval \((x,x+h]\) for \(x^\theta\leq h\leq x\) being asymptotic to \(h/\log x,\) was first established by Hoheisel for some \(0<\theta<1,\) following from the explicit formula and an explicit (weaker) version of the zero-density estimate for the Riemann zeta function as in \eqref{zerodensity}.
 The  exponent \((30+\varepsilon)/13\) in \eqref{zerodensity} was obtained in recent breakthrough work of Guth and Maynard \cite{guthmayn}, leading to \(\theta>17/30\) being permissible. This improved the long-standing exponent of \((12+\varepsilon)/5\) due to Huxley \cite{hux} which allowed \(\theta>7/12,\) as well as a result of Heath-Brown \cite{heath} allowing \(\theta=7/12-o(1).\)
 
In some applications, such as in upcoming work of the author \cite{you}, it is desirable to have an error term which decays far faster than that in Theorem \ref{main}. We are able to do this when we allow the main term to be replaced by a more complicated expression which we define later on.

\begin{theorem}\label{mainfunc} Let \(17/30<\theta\leq 1.\) There exists a constant  \(c=c(\theta)>0,\)  as well as a  function \(f_\theta(x,y)\) which is independent of \(h,\) such that the estimate \[\frac{\Psi(x+h,y)-\Psi(x,y)}{h}=
f_\theta(x,y)+O_\theta(E),\] with the error term  \[E=\frac{\Psi(x,y)}{x}\left(e^{-c((\log x)/\log \log x)^{1/3}}+\frac{h \log(u+1) }{x \log y}+\frac{y}{x\log x}\right),\]  holds uniformly in the range \eqref{yandh}.  \end{theorem} 
Naturally, we obtain a stronger result if we assume the Riemann Hypothesis.
\begin{theorem}\label{mainriemann} Assume the Riemann Hypothesis is true. Let \(1/2<\theta\leq 1.\) There exists a constant  \(K=K(\theta)>2,\)  as well as a  function \(f_\theta(x,y)\) which is independent of \(h,\) such that the estimates
\[\frac{\Psi(x+h,y)-\Psi(x,y)}{h}=\alpha\frac{ \Psi(x,y) }{x}\left(1+O_\theta\left(\min\left\{\frac{1}{u}+\frac{h\log(u+1)}{x\log y},\frac{\log(u+1)}{\log y}\right\}\right)\right)\] and 
\[\frac{\Psi(x+h,y)-\Psi(x,y)}{h}=
f_\theta(x,y)+ O_\theta\left(\frac{\Psi(x,y)}{x}\left(\frac{1}{x^{1/K}}+\frac{h\log(u+1)}{x\log y}+\frac{y}{x\log x}\right)\right)\]
hold uniformly for \[ x^{\theta}\leq h\leq x \quad \textrm{ and } \quad
(\log  x)^K \leq y \leq 2x.\]

For any \(L>2,\) the result also holds for \((\log x)^{L}\leq y \leq 2x\) and some \(\theta=\theta(L)<1.\) 
\end{theorem}
Our method naturally extends to estimating character sums over smooth numbers in short intervals, as well as to smooth numbers which are coprime to some fixed integer. This is used in  upcoming work on prescribing the digits of smooth numbers \cite{you}.
\subsection{Dependence on the interval size} \label{dependence} Theorem \ref{mainfunc} can deliver quickly decaying error terms except when \(y\) (or \(h\)) is very large relative to \(x,\) and one can ask whether this is avoidable.
Suppose for example that  \(y\geq x/2\) and \(x^{9/10}\leq h\leq 2x^{9/10}.\)  Then  counting \(y\)-smooth numbers in the interval \((x,x+h]\) is equivalent to counting all numbers in the interval, and then excluding those which are of the form  \(p\) or \(2p,\) for primes \(p>y.\)  That is,
\[\Psi(x+h,y)-\Psi(x,y)= \sum_{n\in(x, x+h]}1-\sum_{p\in(x,x+h],~ p>y}1 - \sum_{p\in(x/2,(x+h)/2],~ p>y}1.\]
Now let  \(y=(x+x^{9/10})/2\) and \(h_1=2x^{9/10}.\) Using Huxley's  prime number theorem in short intervals,  \[\Psi(x+h_1,y)-\Psi(x,y)=h_1\left(1-\frac{1}{\log x}-\frac{1}{4\log y} +o\left(\frac{1}{\log x}\right)\right).\] On the other hand, if \(h_2=x^{9/10},\) then since \(2p>2y= x+h_2\) the range of the third sum is empty, so  \[\Psi(x+h_2,y)-\Psi(x,y)=h_2\left(1-\frac{1}{\log x} +o\left(\frac{1}{\log x}\right)\right).\]  
This example demonstrates that the density of smooth numbers in \((x,x+h]\) is not always independent of \(h,\) at least when we require error terms as small as \(o(1/\log x).\) 

\subsection{Background on smooth numbers}\label{background}
For this subsection, let \(2\leq y\leq x.\) 
The zeta function for  \(y\)-smooth numbers, which  factorises as an Euler product, is \[\zeta(s,y)=\sum_{n \textrm{ is } y\textrm{-smooth}}\frac{1}{n^s}=\prod_{p\leq y}\left(1+\frac{1}{p^s}+\frac{1}{p^{2s}}+\dots\right)=\prod_{p\leq y}\left(1-\frac{1}{p^s}\right)^{-1}.\] This is valid for  \(s\in \mathbb{C}\) with \(\Re(s)>0,\) as the finitely many geometric series converge here.

The set of \(y\)-smooth numbers up to \(x\) has an associated real parameter \(\alpha=\alpha(x,y)\) known as the \emph{saddle point}. This is defined as the unique  \(\sigma>0\) which minimises
\(x^\sigma\zeta(\sigma,y).\) A short argument known as Rankin's trick asserts that \[\Psi(x,y)=\sum_{n\leq x \textrm{ is } y\textrm{-smooth}} 1\leq \sum_{n \textrm{ is } y\textrm{-smooth}}\left(\frac{x}{n}\right)^\alpha =  x^\alpha\zeta(\alpha,y).\]
Hildebrand and Tenenbaum \cite[Thm.~ 1]{hilten}, using their celebrated saddle point method,  proved an asymptotic formula for \(\Psi(x,y),\) which in particular shows that the Rankin-type bound is actually close to optimal, in the sense that  \begin{equation}\label{hildtenboundzeta} x^\alpha\zeta(\alpha,y)\ll\Psi(x,y)\sqrt{\log x \log y}.\end{equation}
 As \(\alpha\) minimises \(x^\sigma\zeta(\sigma,y),\) if we take the logarithm, differentiate (see \eqref{logdiff}), and set the derivative to \(0,\) then (provided that \(\alpha>1/2,\) which is certainly the case for us) we have
\[\log x =\sum_{p\leq y} \frac{\log p}{p^\alpha -1}=\sum_{n\leq y}\frac{\Lambda(n)}{n^\alpha}+O(1).\]
As noted by La Bret\`eche and Tenenbaum \cite[\S3]{bretten},
for all sufficiently large \(x\) we have \begin{equation}\label{alphaupper}\alpha<1\end{equation} since classically (see e.g.~\cite[Eqn.~(4.21)]{rossch}) \[\sum_{p\leq x} \frac{\log p}{p -1}=\log x-\gamma+o(1)\] as \(x \rightarrow \infty,\) where \(\gamma=0.5772\dots\) denotes Euler's constant.  
An approximate formula for the saddle point was obtained by Hildebrand and Tenenbaum \cite[Lemma 2]{hilten}  \begin{equation}\label{alpha}
\alpha=1-\frac{\log (u\log( u+1))+O(1)}{\log y},\end{equation}  for \(\log x<y\leq x,\)    recalling that \(u=(\log x)/\log y.\)  With \(y\leq x\) in the range \eqref{yandh}, observe that \(\alpha\) is very close to 1.  

From the triangle inequality, it is clear that \(|\zeta(\alpha+it,y)|\leq \zeta(\alpha,y)\) (and hence, as \(\zeta(\alpha,y)\) is the minimum along the positive real axis, \(\alpha\) is called the \emph{saddle point}). 
Hildebrand and Tenenbaum  \cite[Lemma 8]{hilten} provided quantitative bounds, which were pivotal to their saddle point method, showing for instance that for  \(\log x< y\leq x\) suitably large in terms of a fixed \(\varepsilon>0,\) one has
 \begin{equation}\label{quantbounds}\frac{|\zeta(\alpha+it,y)|}{\zeta(\alpha,y)}\ll_\varepsilon e^{-c_0u t^2/(t^2+(1-\alpha)^2)}\end{equation} for \(1/\log y\leq |t|\leq e^{(\log y)^{3/2-\varepsilon}}\) and some constant \(c_0>0.\)
  
It was shown  by Canfield, Erd\H{o}s, and Pomerance \cite{canerdpom} that 
\begin{equation}\label{densityinu}\Psi(x,y)\sim \frac{x}{u^{u(1+o(1))}},\end{equation}  as \(u\rightarrow \infty,\) whenever \(y\geq (\log x)^{1+\varepsilon}.\) For the smallest \(y\) in the range \eqref{yandh}, we have \(\Psi(x,y)\sim xe^{-(\log x)^{1/3+o(1)}}\) as \(x\rightarrow \infty,\) which  it is worth noting is  much sparser than the primes. Clearly, \(y\)-smooth numbers are more frequent as \(y\) increases, thus for all \(y\) satisfying \eqref{yandh} we have \begin{equation}\label{density} \Psi(x,y)\gg \frac{x}{e^{(\log x)^{0.34}}}.
\end{equation} 

\subsection{Outline of argument} 
By Perron's formula (see e.g.~\cite[Thm.~5.2, Cor.~ 5.3]{montvaug}) twice, we express the count of smooth numbers in an interval in terms of a vertical integral in the complex plane,
\begin{equation}\label{perron}\Psi(x+h,y)-\Psi(x,y)=\frac{1}{2\pi i}\int_{\alpha-iH}^{\alpha+iH}\zeta(s,y)\frac{(x+h)^s-x^s}{s} ~\mathrm{d}s+O\left(\frac{x\log x}{H}\right),\end{equation}
for \(y,h,H\in[2,x].\) By Cauchy's residue theorem, we move the integral onto a particular contour \(\Gamma\) to obtain
\[\frac{1}{2\pi i}\int_\Gamma\zeta(s,y)\frac{(x+h)^s-x^s}{s} ~\mathrm{d}s,\] as well as an integral over two horizontal lines at imaginary height \(\pm H\) connecting \(\Gamma\) to the original line. The contour \(\Gamma\) is chosen carefully to satisfy certain properties: near the real axis it remains on the line \(\Re(s)=\alpha;\) it then follows an exponential path leftwards; it narrowly avoids zeros of the zeta function; and it moves along the line \(\Re(s)=(1+\varepsilon)/2\) when no zeros are nearby.  The integrand is holomorphic for \(\Re(s)>0,\) so we take on no residues in the process. 

The portion of the integral near the real axis is used to formulate a main term. In order to obtain a saving elsewhere at large imaginary height, we show that \(\zeta(s,y)\) does not become too large on \(\Gamma.\) The Euler product structure of \(\zeta(s,y)\) reveals that the crucial estimate we need to show is that \begin{equation}\label{crucialest}\frac{1}{\log x}\sum_{n\leq y}\frac{\Lambda(n)}{n^s}\end{equation} is small for \(s\) on or to the right of \(\Gamma,\) but with \(\Re(s)\leq \alpha.\) We do this by another contour integral depending on the size of the local zero-free region.
 
As one might have guessed, our lower  bound for \(y\) in \eqref{yandh} is ultimately a consequence of the best available zero-free region of the Riemann zeta function \(\zeta(s)\) by Vinogradov and Korobov:
\[\Re(s)>1-\lambda \quad  \textrm{ and } \quad |\Im(s)|\leq x,\] where here and throughout, for a suitably small absolute constant \(c_1>0,\) we write \begin{equation}\label{lambdadef}\lambda= \frac{c_1}{(\log x)^{2/3}(\log \log x)^{1/3} }.\end{equation}
We  choose \(y\) large enough that \(\alpha\) is well within this zero-free region. Indeed, for \(y\leq x\) satisfying  \eqref{yandh} we have 
\begin{equation}\label{alphalambda}1-\alpha=\frac{\log(u\log(u+1))+O(1)}{\log y}\leq \frac{\log(u\log(u+1))+O(1)}{C(\log x)^{2/3}(\log \log x)^{4/3}}\leq \varepsilon\lambda\end{equation} provided we choose \(C\geq 100/(\varepsilon c_1 ),\) for instance.

This method only delivers a significant enough saving if the  contour \(\Gamma\) can be shown to reside sufficiently far left, at least on average. Since we do not know the real parts of every individual zero of the zeta function,  we utilise some knowledge of their statistical placement known by the log-free zero-density estimate of Guth and Maynard \cite{guthmayn}:
\begin{equation}\label{zerodensity}\sum_{\zeta(\rho)=0: ~\rho=\beta+i\gamma,~ \beta\geq \sigma,~ |\gamma|\leq H}1\ll_\varepsilon H^{(1-\sigma)(30+\varepsilon)/13}.\end{equation}
The \(30/13\) in the exponent is ultimately responsible for  our restriction that \(\theta>17/30.\) If instead we were to use the the weaker zero-density result of Huxley \cite{hux} throughout this paper, then we would attain the weaker bound of \(\theta>7/12.\)

Obstacles to this method emerge when \(y\) becomes large, so we circumvent them with several analytic techniques. Among these is the adoption of a smooth majorant on the smooth numbers, leading to a weighted version of \eqref{crucialest}.

 However, for \(y\) very close to \(x,\) our version of the saddle point method appears to break down entirely. We therefore adopt a different approach having more in common with that of Hildebrand  \cite{hild2}.  This transforms the problem from counting smooth numbers in short intervals  into one of counting primes in short intervals, by observing that at this level of smoothness we need only exclude numbers with a few large prime factors. We require an estimate for the count of primes in short intervals with quickly decaying error terms. In order to attain this, we first weight the primes, use the explicit formula, and then use the zero-density estimate of Guth and Maynard. This once again leads to the \(17/30\) bound for \(\theta.\) 
 
\subsection{Saddle point method and related work} A typical method of obtaining short interval results for primes is via the explicit formula. Unfortunately, the smooth numbers appear not to enjoy such a formula. An alternative approach is via the Hooley--Huxley contour,  appearing in work of Ramachandra \cite{ram},  and involves dissecting and shifting the Perron integral to the left  of the \(1\)-line in a way that stops short of hitting zeros of \(\zeta(s).\) One then employs the zero-density result to the savings corresponding to the location of the contour. A delicate choice of contour passing through zero-free regions, paired with zero-density estimates, is a key component of  our proof.

The majority of this paper follows the framework of the saddle point method, which first appeared in the landmark work of Hildebrand and Tenenbaum \cite{hilten}, achieving an asymptotic formula for the count of smooth numbers. Additionally, the method was able to deliver the short interval result quoted earlier, which we discuss here. After expressing the count as in \eqref{perron}, a saving in \(x^s\zeta(s,y)\) over the trivial bound is fundamental for \(s=\alpha+it\) and \(1/\log y\leq |t|\leq H.\)  By exploiting the Euler product structure, a saving of at most \(e^{-u}\) is attained, as in \eqref{quantbounds}, provided
\(H\leq e^{(\log y)^{3/2}}\) (due to the Vinogradov--Korobov zero-free region). This saving can only be sufficient if it is greater than the length of the Perron  integral, meaning \(H\leq e^u.\) Upon equalising these two restrictions, and noting that the error in Perron's formula at least requires that \(H\geq x/h,\) one arrives at the lower bound of applicability \(h\geq x/ e^{(\log x)^{3/5}}.\)

Of particular relevance to our proofs is work of Soundararajan \cite{sound},  later built upon by Harper \cite{harp2}, studying the equidistribution of smooth numbers in an arithmetic progression. A key insight featured in these works is that a saving is instead possible from the \(x^s\) term by shifting left from \(\alpha+it\) into a (rectangular) zero-free region, and by showing that \(\zeta(s,y),\) or rather an analogue involving Dirichlet characters, does not become too large in  the process. The latter task is done by considering  expressions similar in  form to \eqref{crucialest}. This approach is then combined with statistical behaviour of the zeros of a Dirichlet \(L\)-function by zero-density estimates.  Harper \cite{harp3} went on to prove results for smooth numbers in arithmetic progressions on average, analogous to the Bombieri--Vinogradov theorem for example, by further developing the saddle point method with several new ideas, including a delicate dissection of contributions from zeros of families of \(L\)-functions, as well as employing the large sieve (though our paper does not use the latter of these tools). Many of these ideas transfer to the short interval setting.

\subsection{Organisation}The plan for the remainder of the article is as follows.
For small \(\varepsilon>0,\) we let \(\theta\geq 17/30+10\varepsilon,\) and set a suitably large \(C\asymp 1/\varepsilon.\)   The objective of Section \ref{shiftingleft} is to demonstrate a power saving for \(x^s\zeta(s,y)\) when shifting left into a local zero-free region. In Section \ref{contour} we move our Perron integral onto a contour \(\Gamma,\) left of the vertical line at \(\Re(s)=\alpha,\) and prove that sufficient  savings occur away from the real axis when \(y\leq x^\varepsilon\) (where the constant introduced in Theorem \ref{mainfunc} has size \(c\asymp \varepsilon\)).  We employ a smooth weight function in Section \ref{smoothing} to improve this to \(y\leq x^{1-2\varepsilon}\) (now with \(c\asymp \varepsilon^2\)). This gives sufficient error terms when \(y\leq x^{1/2-\varepsilon},\)  or when \(x^\theta\leq h\leq x^{0.99},\) say.  
In Section \ref{largey}, which may essentially be read independently of earlier sections, we use a direct argument counting primes in short intervals to cover the range \(x^{1-2\varepsilon}\leq y\leq 2x\) (such that \(c\asymp \varepsilon\)), as well as when \(x^{1/2-\varepsilon} \leq y\leq x^{1-2\varepsilon}\) and \( x^{0.99}\leq h\leq x.\) In Section \ref{proofthms}, we combine the results on different ranges to complete the proof of Theorem \ref{mainfunc}. The proof of Theorem \ref{main} quickly follows by proving the main term \(f_\theta(x,y)\) approximates \(\Psi(x,y)/x.\)

 Finally, Section \ref{riemannsection} contains the proof of Theorem \ref{mainriemann}, in which we revisit and simplify our method under the assumption of the Riemann Hypothesis (with \(K\asymp 1/\varepsilon^2,\) and in the smallest range of \(y\) with  \(L= 2+ 10\varepsilon\) and \(1-\theta\asymp \varepsilon\) instead).

\subsection*{Notation} We adopt standard asymptotic notation. We write \(f\ll g\) or \(f=O(g)\) to mean \(|f|\leq C g\) for some constant \(C>0.\) Let \(f\asymp g\) signify \(f\ll g \ll f.\) For \(x\rightarrow \infty,\) we also write \(f=o(g)\) if \(f/g\rightarrow 0,\) and \(f\sim g\) if \(f/g\rightarrow 1.\) 

When bounding complex integrals, we use the convenient notation \(\int_r f(z) ~|\mathrm{d}z|\) to mean \(\int_a^bf(r(t))|r'(t)|~\mathrm{d}t\) for a  differentiable path \(r,\) and extend the notation for piecewise differentiable paths linearly as one would expect.
\subsection*{Acknowledgements}
 The author would like to thank Adam Harper for many helpful discussions and encouragement. 
The author is supported by the Warwick Mathematics Institute Centre for Doctoral Training, and gratefully acknowledges funding by the Swinnerton-Dyer scholarship. 
\subsection*{Rights Retention} For the purpose of open access, the author has applied a Creative Commons Attribution (CC-BY) licence to any Author Accepted Manuscript version arising
from this submission.

\section{Shifting left}\label{shiftingleft}
Recall that the defining property of \(\alpha\) is as a saddle point:    for all \(\sigma>0,\) one has
\[x^\sigma\zeta(\sigma,y)\geq x^\alpha\zeta(\alpha,y).\]  In this section we show that this inequality can be reversed, and with a power saving in \(x,\) when \(\sigma\) is replaced by \(\sigma+it\in \mathbb{C}\) in some zero-free region far away from the real axis. To do this, we first establish some preparatory results.  

\begin{lemma}[Bounded logarithmic derivative in zero-free region] Let \(\varepsilon>0\) be small, and let  \(0< \eta\leq1/2.\)  Write \(s_0=\sigma_0+it_0\neq 1,\) where \(\sigma_0\geq 1-(1-\varepsilon)\eta.\) Suppose there are no zeros of \(\zeta(s)\) in the region 
\[\Re(s)>1-\eta \quad \textrm{and} \quad \Im(s)\in (t_0-1,t_0+1).\]
 Then
 \[\left|\frac{-\zeta'}{\zeta}(s_0)\right|\leq \frac{1}{|s_0-1|}+O_\varepsilon \left(\log (|t_0|+2)\right).\]
\end{lemma}
\begin{proof} If \(|t_0|< 2\) then, from the Laurent expansion of \(-\zeta'/\zeta,\) we find that
\[\left|\frac{-\zeta'}{\zeta}(s_0)\right|\leq \frac{1}{|s_0-1|}+O(1).\]
Now assume \(|t_0|\geq 2.\) Consider the partial fraction expansion  when \(|t|\geq 2,\) given by 
\[\frac{-\zeta'}{\zeta}(s)=-\sum_\rho \frac{1}{s-\rho}+O(\log |t|),\] where the sum is over zeros \(\rho=\beta+i\gamma\) such that \(\gamma\in(t-1,t+1)\) (see e.g.~ \cite[\S 15]{dav}).
Set \(s_1=1+it_0.\) We have the standard bound on the \(1\)-line 
\begin{equation}\label{logbound}\frac{-\zeta'}{\zeta}(s_1)\ll \log |t_0|,\end{equation} (see e.g.~\cite[Thm 6.7]{montvaug}) and thus from the partial fraction expansion
\begin{equation}\label{logbound2}\Re\sum_\rho\frac{1}{s_1-\rho}\ll  \log |t_0|.\end{equation}

 Now consider the difference \begin{equation}\label{diff}\frac{-\zeta'}{\zeta}(s_0)-\frac{-\zeta'}{\zeta}(s_1)=-\sum_\rho \left(\frac{1}{s_0-\rho}-\frac{1}{s_1-\rho}\right)+O(\log |t_0|).\end{equation} Our hypothesised zero-free region implies \(|s_0-\rho|\asymp_\varepsilon|s_1-\rho|,\) and \(|s_1-s_0|=1-\sigma_0\leq 1-\beta= \Re (s_1-\rho).\) Each summand is   therefore at most
\[\frac{|s_1-s_0|}{|s_0-\rho||s_1-\rho|}\ll_\varepsilon\Re\frac{s_1-\rho}{|s_1-\rho|^2}=\Re \frac{1}{s_1-\rho}.\] From the earlier estimates \eqref{logbound} and \eqref{logbound2}, we see that \eqref{diff} yields
\[\frac{-\zeta'}{\zeta}(s_0)=\frac{-\zeta'}{\zeta}(s_1)+O_\varepsilon\left(\Re \sum_\rho\frac{1}{s_1-\rho}\right)+O(\log|t_0|)\ll_\varepsilon \log |t_0|,\] completing the proof.
\end{proof}
In the following argument in complex analysis, we  improve the \(\log |t_0|\) bound by interpolating between this and the stronger bound  \((\log|t_0|)^{2/3+\varepsilon}\) to the right of the Vinogradov--Korobov zero-free region.
\begin{lemma}[Improved bound]\label{improved}
Let \(\varepsilon>0\) be small, and let  \(0< \eta\leq1/2.\)  Write \(s_0=\sigma_0+it_0\neq 1,\) where \(\sigma_0\geq 1-(1-\varepsilon)\eta.\) Suppose there are no zeros of \(\zeta(s)\) in the region 
\[\Re(s)>1-\eta \quad \textrm{and} \quad \Im(s)\in (t_0-2,t_0+2).\] Then
\[\left|\frac{-\zeta'}{\zeta}(s_0)\right|
 \leq \frac{1}{|s_0-1|}+O_\varepsilon ((\log (|t_0|+2))^b),\] 
 for some constant \(b=b(\varepsilon)<1.\)
\end{lemma}
\begin{proof}
Assume \(|t_0|\geq 2,\) otherwise the result is easy to deduce as before. Consider circles \(C_1,C_2,C_3\) each having centre \(3/2+it_0\) and respective radii \(r_1=1/2, r_2= 1/2+(1-\varepsilon)\eta,\) and \(r_3= 1/2+(1-\varepsilon/2)\eta.\) Write \(M_j\) for the maximum value that \(|-\zeta'(s)/\zeta(s)|\) attains on \(C_j,\) which is finite by the assumed zero-free region. We have \(M_1\ll (\log |t_0|)^{2/3}(\log \log |t_0|)^{1/3}\) (see e.g.~ \cite[p.\,135]{titc}), and \(M_3\ll_\varepsilon \log |t_0|\) from the previous lemma (with \(\varepsilon\) scaled by \(1/2\)).

 By Hadamard's three circle theorem (as in e.g.~\cite[p.\,337]{titc}), 
\[M_2\leq M_1^{1-a}M_3^a\] where \[a=\frac{\log (r_2/r_1)}{\log (r_3/r_1)}=\frac{\log(1+2(1-\varepsilon)\eta)}{\log(1+2(1-\varepsilon/2)\eta)}.\] Observe that \(a\) is an increasing function of \(\eta> 0,\)
so is maximised at \(\eta=1/2.\) Thus \(a\leq \log(2-\varepsilon)/\log(2-\varepsilon/2)<1,\) and the result follows. 
\end{proof}

Consider a trivial upper bound from the triangle inequality  \[\left|\sum_{n\leq y}\frac{\Lambda(n)}{n^{\alpha+it}}\right|\leq\sum_{n\leq y}\frac{\Lambda(n)}{n^{\alpha}}= \sum_{p\leq y}\frac{\log p}{p^\alpha-1}+O(1)= \log x+O(1).\]
The trivial bound is even worse when replacing \(\alpha+it\) with  \(\sigma+it\) such that \(\sigma\leq \alpha.\) In order to obtain a saving, we must make use of the oscillations from \(n^{-it}.\) To do this we go via Perron's formula, provided we have a suitably large zero-free region around \(\sigma+it.\)
\begin{lemma}[Shifting into local zero-free region yields cancellation]\label{shifting} Let \(z\geq 2.\)  Let \(\varepsilon>0\) be small, and let  \(0< \eta\leq1/2.\)  Write \(s_0=\sigma_0+it_0 \neq 1,\) where \(1-(1-\varepsilon)\eta\leq \sigma_0<1+1/\log z.\) Suppose there are no zeros of \(\zeta(s)\) in the region 
\[\Re(s)>1-\eta \quad \textrm{and} \quad \Im(s)\in (t_0-10z^\eta,t_0+10z^\eta).\] Then for some  \(b=b(\varepsilon)<1,\) we have
\begin{align*}
\left|\sum_{n \leq z} \frac{\Lambda(n)}{n^{s_0}}\right|\leq \frac{z^{1-\sigma_0}+1}{|1-s_0|}&+O_\varepsilon\left((\log (|t_0|+2))^b\right)+O\left(\frac{(\log z)^2}{z^{\varepsilon\eta}}\right) \\
&+O_\varepsilon\left(\frac{((\log(|t_0|+z^\eta+2))^b+1/\eta)\log(z^\eta/\eta)}{z^{\varepsilon\eta/2}}\right).
\end{align*}  
\end{lemma} 

\begin{proof}
If \(2\leq T\leq 10z,\) then by Perron's formula (see e.g.~\cite[Thm.~5.2, Cor.~5.3]{montvaug}) \[\sum_{n\leq z} \frac{\Lambda(n)}{n^{s_0}}=\frac{1}{2\pi i}\int_{1-\sigma_0+1/\log z-iT}^{1-\sigma_0+1/\log z+iT}\frac{-\zeta'}{\zeta}(s+s_0)\frac{z^s}{s}~\textrm{d}s+O\left(\frac{z^{1-\sigma_0}(\log z)^2}{T}\right).\]
A change of variable for the integral, and also insisting that \(T\geq  z^{\eta},\) leads to \[\frac{1}{2\pi i}\int_{1+1/\log z+it_0-iT}^{1+1/\log z+it_0+iT}\frac{-\zeta'}{\zeta}(s)\frac{z^{s-s_0}}{s-s_0}~\textrm{d}s+O\left(\frac{(\log z)^2}{z^{\varepsilon\eta}}\right).\]
We shift the contour left to \(\Re(s)=r=1-(1-\varepsilon/2)\eta.\) There are two cases to consider. 
  
{\it Case 1.} Assume \(|t_0|\geq 3z^\eta.\) Set \(T=2z^\eta.\) We do not encounter the simple pole of \(-\zeta'(s)/\zeta(s)\) at \(s=1.\) By Cauchy's residue theorem, the integral is equal to \begin{equation}\label{cauchy}
\frac{-\zeta'}{\zeta}(s_0)+\frac{1}{2\pi i}\left(\int_{r+it_0-iT}^{r+it_0+iT}+\int_{1+1/\log z+it_0-iT}^{r+it_0-iT}+\int_{r+it_0+iT}^{1+1/\log z+it_0+iT}\right)\frac{-\zeta'}{\zeta}(s)\frac{z^{s-s_0}}{s-s_0}~\textrm{d}s.\end{equation}
  We have \(|-\zeta'(s_0)/\zeta(s_0)|\leq 1/|s_0-1|+O_\varepsilon((\log(t_0+2))^b)\) by Lemma \ref{improved}.

{\it Case 2.} Assume \(|t_0|<3z^\eta.\) Set \(T=4z^{\eta}.\) We apply Cauchy's residue theorem exactly as before,  but we do encounter the simple pole of \(-\zeta'(s)/\zeta(s)\) at \(s=1,\) which gives \eqref{cauchy} plus an additional term
\[\frac{z^{1-s_0}}{1-s_0}.\] 

In either case, we now handle \eqref{cauchy}. 
  By  Lemma \ref{improved}, the upper horizontal integral is crudely bounded by 
\[\ll \max_{\sigma\in[r,1+1/\log z]}\left|\frac{-\zeta'}{\zeta}(\sigma+it_0+ iT)\right|\frac{z^{1-\sigma_0}}{T}\ll_\varepsilon \frac{(\log(|t_0|+z^\eta+2))^b}{z^{\varepsilon\eta}}.\]  The lower horizontal integral is bounded similarly.  The vertical integral is 
\[\ll z^{r-\sigma_0} \max_{t\in[-T,T]}\left|\frac{-\zeta'}{\zeta}(r+it_0+ it)\right|\int_{-T}^T\frac{1}{|r-\sigma_0+it|} ~\mathrm{d}t.\]
We have \(|r-\sigma_0+it|\geq \max\{\varepsilon\eta/2,|t|\},\) since we can consider the real and imaginary parts separately.  Again using Lemma \ref{improved}, the above expression is then
  \begin{align*}
&\ll_\varepsilon\frac{(\log(|t_0|+T+2))^b+ 1/\eta }{z^{\varepsilon\eta/2}}\left(\int_{\eta}^{T} \frac{1}{t}~\textrm{d}t+\int_{-\eta}^{\eta} \frac{1}{\eta}~\textrm{d}t\right)\\
&\ll\frac{(\log(|t_0|+z^\eta+2))^b+ 1/\eta }{z^{\varepsilon\eta/2} }\left(\log (z^\eta)+\log (1/\eta)+1\right)
\end{align*}
as desired.
\end{proof}
In the next lemma, we achieve the objective that we stated at the start of this section: we exhibit a power saving in \(x^s\zeta(s,y)\) as \(s\) travels far enough left from \(s_2=\alpha+it_0\) into a local zero free-region, provided \(t_0\) is large enough. 
Recall from \eqref{lambdadef} that the  width of the zero-free region of Vinogradov and Korobov up to imaginary height \(x\)  is given by \[\lambda= \frac{c_1}{(\log x)^{2/3}(\log \log x)^{1/3} }\] for a suitably small absolute constant \(c_1>0.\)
\begin{lemma}[Power saving from shifting left locally]\label{powersave} Let \(\varepsilon>0\) be small, and let  \( \lambda\leq \eta\leq 1/2.\) 
Write \(s_1=\sigma_1+it_0,\) and \(s_2=\sigma_2+it_0,\) where    \(1-(1-\varepsilon)\eta\leq \sigma_1<\sigma_2<1+1/\log y\)  and \[\frac{\log(u+1)}{ \varepsilon^2(\log y) y^{1-\alpha}}y^{1-\sigma_1}\leq |t_0|\leq x.\]  Suppose there are no zeros of \(\zeta(s)\) in the region \[\Re(s)>1-\eta \quad \textrm{and} \quad \Im(s)\in(t_0-10y^{\eta},t_0+10y^{\eta}).\] 
Then for \(y\leq x\) in the range \eqref{yandh},  and \(x\)  sufficiently large in terms of \(\varepsilon,\) we have
\[|x^{s_1}\zeta(s_1,y)|\ll \frac{|x^{s_2}\zeta(s_2,y)|}{x^{(1-\varepsilon)(\sigma_2-\sigma_1)}}.\]
\end{lemma}

\begin{proof} 
Recall the Euler product \(\zeta(s,y)=\prod_{p\leq y}\left(1-{p^{-s}}\right)^{-1}\) for \(\Re(s)>0,\) so in particular \(0<|\zeta(s,y)|<\infty.\)  Thus  \(\log \left|\zeta(s_1,y)/\zeta(s_2,y)\right|\) is finite and equals   \[\Re\log \frac{\zeta(s_1,y)}{\zeta(s_2,y)}=\Re\int_{s_2}^{s_1}\frac{\zeta'}{\zeta}(s,y)~\mathrm{d}s\leq(\sigma_2-\sigma_1)\sup_{\substack{\Re(s)\geq \sigma_1 \\ \Im(s)=t_0}}\left|\frac{\zeta'}{\zeta}(s,y)\right|.\]
 Using the product rule for differentiation, we have 
\begin{equation}\label{logdiff}
\frac{-\zeta'}{\zeta}(s,y)=\sum_{p\leq y}\frac{\log p}{p^{s}-1}=\sum_{n \leq y} \frac{\Lambda(n)}{n^{s}}+O\left(\sum_{\substack{p\leq y\\ k\geq 2}}\frac{\log p}{\left|p^{ks}\right|}\right).\end{equation}
We may bound the error term, if \(\Re(s)\geq (1+\varepsilon)/2,\) by summing a geometric progression 
\[\sum_{\substack{p\leq y\\ k\geq 2}}\frac{\log p}{p^{ks}}\ll \sum_{p\leq y}\frac{\log p}{|p^{2s}|}\ll_\varepsilon 1.\]
 
For \(x\) large enough in terms of \(\varepsilon,\) we claim that for \(s\) in the above supremum we have  \[\left|\sum_{n \leq y} \frac{\Lambda(n)}{n^{s}}\right|\leq \frac{\varepsilon\log x}{2},\] which suffices to prove the lemma. 
First  observe that \[y^{\varepsilon\eta}\geq y^{\varepsilon\lambda}= \exp\left(\frac{\varepsilon c_1\log y}{(\log x)^{2/3}(\log \log x)^{1/3}}\right)\geq (\log x)^{C\varepsilon c_1}\] and choose \(C\geq 100/(\varepsilon c_1).\) We then see from Lemma \ref{shifting} that
\[\left|\sum_{n\leq y}\frac{\Lambda(n)}{n^s}\right|\leq \frac{y^{1-\sigma}+1}{|1-s|}+O_\varepsilon((\log(|t_0|+2))^b)+O_\varepsilon(1).\]
Certainly  we can bound the term \[O_\varepsilon((\log(|t_0|+2))^b)\leq \frac{\varepsilon\log x}{ 10}\] for \(x\) sufficiently large in terms of \(\varepsilon.\)  Next, using the approximate formula \eqref{alpha} for \(\alpha,\)  \[\left|\frac{y^{1-\sigma}}{1-s}\right|\leq \frac{y^{1-\sigma_1}}{|t_0|}\leq \frac{\varepsilon^2 (\log y) y^{1-\alpha}}{\log(u+1)} \leq  \frac{\varepsilon^2 (\log y) e^{O(1)}u\log (u+1)}{\log(u+1)}\leq \frac{\varepsilon\log x}{10},\] as \(\varepsilon>0\) is suitably small. 
Similarly, 
\[\left|\frac{1}{1-s}\right|\leq \frac{ey^{1-\sigma_1}}{|t_0|}\leq  \frac{\varepsilon\log x}{10}.\]
 Assembling these bounds proves the claim, and hence the lemma. 
\end{proof}
\section{Contour}\label{contour}
Let \(\varepsilon>0\) be small. For each zero \(\rho=1-\nu+i\gamma\) of \(\zeta(s)\) with \(\nu\leq 1/2,\) consider the (connected) path \(\Gamma_\rho=\Gamma_\rho'\cup\Gamma_\rho^+\cup\Gamma_\rho^-,\)  given by the union of a vertical line \[\Gamma_\rho'(t)=1-(1-\varepsilon)\nu+i(\gamma+10t), \quad \quad t\in[-y^{\nu},y^{\nu}] ,\] and the exponential graphs \[\Gamma_\rho^\pm(\delta)=1-(1-\varepsilon)\delta+i(\gamma\pm 10y^{\delta}), \quad \quad \delta\in\left[\nu, \frac{1}{2}\right] .\] 
Also, let \(\Gamma_0=\Gamma_0'\cup\Gamma_0^+\cup \Gamma_0^-\)  be the union of  the vertical line
\[\Gamma_0'(t)=\alpha+ it,\quad \quad t\in \left[-\frac{\log(u+1)}{\varepsilon^2 (\log y)},\frac{\log(u+1)}{\varepsilon^2 (\log y)}\right]\]
 and the exponential graphs \[\Gamma_0^\pm(\delta)=1-\delta\pm i\frac{\log(u+1)}{\varepsilon^2 (\log y)y^{1-\alpha}}y^\delta,\quad \quad \delta\in \left[1-\alpha,\frac{1-\varepsilon}{2}\right].\] 
For every \(t\) satisfying \( |t|\leq H,\) for some \(H\)  we specify later, select the largest \(\sigma\) such that \(\sigma+it \in \Gamma_0 \cup(\bigcup_\rho \Gamma_\rho),\) and if none exists then set \(\sigma=(1+\varepsilon)/2\). This set of \(s=\sigma+it\) forms a continuous path regarded as travelling up the complex plane, which we denote \(\Gamma\) (and \(\Gamma(t)\) denotes this point \(\sigma+it\)). 

Recall by \eqref{alphaupper} that \(\alpha<1\) for \(x\) sufficiently large. 
 The way we have constructed our path \(\Gamma\) ensures, by Lemma \ref{powersave}, that if \(y\leq x\) is in the range \eqref{yandh}, and if \(x\) is  sufficiently large in terms of \(\varepsilon,\) then every \(s=\sigma+it\) on or to the right of  \(\Gamma,\) but to the left of \(\alpha,\) satisfies \[|x^{s}\zeta(s,y)|\ll \frac{|x^{\alpha+it}\zeta(\alpha+it,y)|}{x^{(1-\varepsilon)(\alpha-\sigma)}}\leq \frac{x^\alpha\zeta(\alpha,y)}{x^{(1-\varepsilon)(\alpha-\sigma)}}.\]  The second inequality is from a simple application of the triangle inequality. 

Recall from Perron's formula \eqref{perron} that for \(2\leq y\leq x\) and \(2\leq H\leq x\) we have \[\Psi(x+h,y)-\Psi(x,y)=\frac{1}{2\pi i}\int_{\alpha-iH}^{\alpha+iH}\zeta(s,y)\frac{(x+h)^s-x^s}{s}~\mathrm{d}s+O\left(\frac{x\log x}{H}\right).\]
By Cauchy's theorem, we alter the path of integration to obtain
\[\frac{1}{2\pi i}\left(\int_{\Gamma\cap\Gamma_0}+\int_{\bigcup_\rho\left(\Gamma\cap\Gamma_\rho\right)}+\int_{\Gamma\setminus \left(\Gamma_0 \cup\left(\bigcup_\rho \Gamma_\rho\right)\right)}+\int_{\substack{\sigma \pm iH\\\Gamma(\pm H)\leq \sigma\leq \alpha}}\right)\zeta(s,y)\frac{(x+h)^s-x^s}{s} ~\mathrm{d}s\] and encounter no poles in the process.  
 Now observe the expansion
\[\left(1+\frac{h}{x}\right)^s=1+s\int_0^{h/x} (1+v)^{s-1} ~\mathrm{d}v=1+s\int_0^{h/x} \left(1+(s-1)\int_0^{v} (1+w)^{s-2} ~\mathrm{d}w \right)~\mathrm{d}v.\] As \(\Re(s)\leq 10,\) we can bound the integrands  \(|(1+v)^{s-1}|\ll 1\) and \(|(1+w)^{s-2}|\ll 1,\)   so
\begin{equation}\label{taylor0} \frac{(x+h)^s-x^s}{s}=x^s\frac{(1+h/x)^s-1}{s}\ll|x^s|\frac{h}{x},\end{equation} and
 \[\frac{(x+h)^s-x^s}{s}=x^s\frac{(1+h/x)^s-1}{s}=x^s\left(\frac{h}{x}+O\left(|s-1|\frac{h^2}{x^2}\right)\right),\] provided \(s\neq 0.\) 

 Applying this, we extract a main term, with an error, from near the real axis
 \[\frac{1}{2\pi i}\frac{h}{x}\int_{\Gamma\cap\Gamma_0}\zeta(s,y)x^s ~\mathrm{d}s+O\left(\int_{\Gamma\cap\Gamma_0}\min\left\{\frac{h^2}{x^2}|s-1|,\frac{h}{x}\right\}|\zeta(s,y)x^s| ~|\mathrm{d}s|\right)\] and we bound the contribution from the zeros and the `remainder' line by
 \[\ll\frac{h}{x}\sum_\rho\int_{\Gamma\cap \Gamma_\rho}|\zeta(s,y)x^s|~|\mathrm{d}s|+\frac{h}{x}\int_{\Gamma\setminus \left(\Gamma_0 \cup\left(\bigcup_\rho \Gamma_\rho\right)\right)}|\zeta(s,y)x^s|~|\mathrm{d}s|.\]
 We also bound the integral over the horizontal path by
 \[\ll \frac{1}{H}\int_{\substack{\sigma \pm iH\\\Gamma(\pm H)\leq \sigma\leq \alpha}}|\zeta(s,y)x^s |~|\mathrm{d}s|,\]
  instead using the triangle inequality  \[\frac{(x+h)^s-x^s}{s}\ll \frac{|x^s|}{H}.\]
 \begin{lemma}[Contribution from zeros]\label{contzeros}  Let \(\varepsilon>0\) be  small, and let \(H\leq x^{13/30-5\varepsilon}.\) Suppose \(y\leq x^\varepsilon\) is in the range \eqref{yandh},  and \(x\)  is sufficiently large in terms of \(\varepsilon.\) Then
 \[\sum_\rho \int_{\Gamma\cap\Gamma_\rho}|\zeta(s,y)x^s|~|\mathrm{d}s|\ll_\varepsilon\Psi(x,y)e^{-c((\log x)/\log \log x)^{1/3}}\]
for a small constant \(c=c(\varepsilon)>0.\) 
 \end{lemma}
 \begin{proof} First consider the contribution from one vertical line
\[
\int_{\Gamma\cap\Gamma_\rho'}|\zeta(s,y)x^s|~|\mathrm{d}s|\ll|\Gamma_\rho'| \frac{\zeta(\alpha,y)x^{\alpha}}{x^{(1-\varepsilon)(\alpha-(1-(1-\varepsilon)\nu))}}\ll \zeta(\alpha,y)x^{\alpha}x^{(1-\varepsilon)(1-\alpha)}\left(\frac{y}{x^{(1-\varepsilon)^2}}\right)^{\nu}.
\]
Similarly, the exponential graph contributes 
\begin{align*}
\int_{\Gamma\cap\Gamma_\rho^+}|\zeta(s,y)x^s|~|\mathrm{d}s|&\ll  \zeta(\alpha,y)x^\alpha  x^{(1-\varepsilon)(1-\alpha)}\int_{\nu}^{1/2} x^{-(1-\varepsilon)^2\delta}\left|\frac{\mathrm{d}}{\mathrm{d}\delta}\Gamma_\rho^+(\delta)\right|~\mathrm{d}\delta\\&\ll \zeta(\alpha,y)x^\alpha x^{(1-\varepsilon)(1-\alpha)}\int_{\nu}^{1/2} x^{-(1-\varepsilon)^2\delta}(\log y )y^{\delta}~\mathrm{d}\delta \\&\ll  \zeta(\alpha,y) x^\alpha  x^{(1-\varepsilon)(1-\alpha)}\frac{\log y}{\log x}\left(\frac{y}{x^{(1-\varepsilon)^2}}\right)^{\nu}.
 \end{align*}
  An identical calculation yields the same bound for \(\Gamma\cap\Gamma_\rho^-.\) 
  
  Recall from \eqref{alphalambda} that \(1-\alpha\leq\varepsilon\lambda\leq \varepsilon\nu,\) and that \(y\leq x^{\varepsilon}.\) Given that \(x^{\alpha}\zeta(\alpha,y)\ll \Psi(x,y)\sqrt{\log x \log y}\) (see \eqref{hildtenboundzeta}), we have
\[\sum_{\rho}\int_{\Gamma\cap \Gamma_\rho} |\zeta(s,y)x^s| ~|\mathrm{d}s|\ll \Psi(x,y)\sqrt{\log x \log y}\sum_\rho x^{-(1-4\varepsilon)\nu}.\]
Since \(|\Gamma_\rho|\ll y^{1/2}\leq x^{\varepsilon/2},\)  the only zeros contributing to \(\Gamma\) are those with imaginary height \(|\gamma|\ll H+x^{\varepsilon/2}.\) 
Consider the contribution to the sum of the \(\rho=1-\nu+i\gamma\) with \[\frac{\ell}{\log x}\leq \nu \leq \frac{\ell+1}{\log x}.\]  
The number of such zeros is \(\ll_\varepsilon (H+x^{\varepsilon/2})^{(30+\varepsilon)(\ell+1)/(13\log x)}\ll x^{(1-5\varepsilon)\ell/\log x}\) by the zero-density estimate \eqref{zerodensity} and the fact that \(H^{(30+\varepsilon)/13}\leq x^{(13/30-5\varepsilon)(30+\varepsilon)/13} \leq  x^{1-5\varepsilon}.\) The contribution from  \(\rho\) in the stated range  is therefore 
\[\ll_\varepsilon x^{(1-5\varepsilon)\ell/\log x} x^{-(1-4\varepsilon)\ell/\log x} = x^{-\varepsilon\ell/\log x}.\]
The number of possible integers \(\ell\)  to cover the whole range of summation is at most \(\log x.\) Also, we may assume that \((\ell+1)/\log x\geq \lambda,\)   for otherwise there are no such \(\rho\) to consider by the Vinogradov--Korobov zero-free region. Therefore \begin{equation}\label{sumzeros}\sum_\rho x^{-(1-4\varepsilon)\nu}\ll (\log x)x^{-\varepsilon\lambda},\end{equation} which is sufficient to prove the result.
 \end{proof}

   \begin{lemma}[Contribution from \(\Gamma_0^{\pm}\)] \label{gamma0} Let \(\varepsilon>0\) be small. Suppose \(y\leq x\) is in the range \eqref{yandh}, and \(x\) is sufficiently large in terms of \(\varepsilon.\) If \(y\leq x^{1/2-\varepsilon}\) then
 \[ \int_{\Gamma\cap\Gamma_0^{\pm}}|s-1||\zeta(s,y)x^s|~|\mathrm{d}s|\ll_\varepsilon\frac{\Psi(x,y)e^{-c_2u}}{\log x},\]  and if \(y\leq x^{1-2\varepsilon}\) then \[ \int_{\Gamma\cap\Gamma_0^{\pm}}|\zeta(s,y)x^s|~|\mathrm{d}s|\ll_\varepsilon \Psi(x,y)e^{-c_2u},\] for a small absolute constant \(c_2>0.\)
 \end{lemma}
 \begin{proof}
For \(s=1-\delta+it\in \Gamma_0^+,\) we claim that \[\delta\leq |t|=\frac{\log(u+1)}{\varepsilon^2 (\log y)y^{1-\alpha}}y^\delta.\] The inequality certainly holds when \(\delta=1-\alpha\) (as \(\varepsilon>0\) is suitably small) from the formula \eqref{alpha}. Furthermore, on \(\Gamma_0^+\) we have  \(\mathrm{d}|t|/\mathrm{d}\delta=\log(u+1)y^{\delta}/(\varepsilon^2 y^{1-\alpha})\geq \log(u+1)/\varepsilon^2\geq 1,\) proving the claim. Thus \(|s-1|\leq2|t|.\)  This reasoning also shows that \(\left|\mathrm{d}\Gamma_0^+/{\mathrm{d}\delta}\right|\leq 1+ \mathrm{d}|t|/{\mathrm{d}\delta}\leq 2(\mathrm{d}|t|/{\mathrm{d}\delta}).\) 
 Therefore
\begin{align*}
\int_{\Gamma\cap\Gamma_0^+}|s-1||\zeta(s,y)x^s|~|\mathrm{d}s|&\ll\int_{1-\delta+it\in\Gamma\cap\Gamma_0^+}|t||\zeta(s,y)x^s|~|\mathrm{d}s|\\ 
& \ll\int_{1-\delta+it\in \Gamma\cap\Gamma_0^+}|t|\left|\frac{\zeta(\alpha+it,y)x^\alpha}{x^{(1-\varepsilon)(\alpha-(1-\delta))}}\right|\left|\frac{\mathrm{d}}{\mathrm{d}\delta}\Gamma_0^+(\delta)\right|~\mathrm{d}\delta.\end{align*}
Using the estimates \eqref{hildtenboundzeta} and \eqref{quantbounds}, the above expression is
\begin{align*}
&\ll \int_{1-\delta+it\in\Gamma\cap\Gamma_0^+}e^{-c_0u}|t|\left|\frac{\zeta(\alpha,y)x^\alpha}{x^{(1-\varepsilon)(\alpha-(1-\delta))}}\right|\left|\frac{\mathrm{d}}{\mathrm{d}\delta}\Gamma_0^+(\delta)\right|~\mathrm{d}\delta 
\\& \ll_\varepsilon\frac{x^{(1-\varepsilon)(1-\alpha)}\zeta(\alpha,y)x^\alpha e^{-c_0u}}{\log y}\int_{1-\delta+it\in\Gamma\cap\Gamma_0^+}\frac{y^{2\delta}}{x^{(1-\varepsilon)\delta}}\left(\frac{\log(u+1)}{y^{1-\alpha}}\right)^2 ~\mathrm{d}\delta
\\&=\frac{x^{(1-\varepsilon)(1-\alpha)}\zeta(\alpha,y)x^\alpha  e^{-c_0u}}{\log y}\left(\frac{\log(u+1)}{y^{1-\alpha}}\right)^2\int_{1-\alpha}^{(1-\varepsilon)/2}\left(\frac{y^2}{x^{1-\varepsilon}}\right)^\delta~\mathrm{d}\delta\\&\ll\Psi(x,y)e^{-c_2u}\int_{1-\alpha}^{(1-\varepsilon)/2}\left(\frac{y^2}{x^{1-\varepsilon}}\right)^{\delta-(1-\alpha)}~\mathrm{d}\delta.
\end{align*}
If \(y^2\leq x^{1-2\varepsilon},\) we evaluate the integral to obtain an overall bound of
\[ \ll_\varepsilon\frac{\Psi(x,y)e^{-c_2u}}{\log x}.\]

 An identical calculation yields the same bound for \(\Gamma\cap\Gamma_0^-.\) 
The second claimed bound is proved very similarly by omitting a multiplicative factor.
 \end{proof}
When \(y\) is relatively large we need a more delicate estimate. The next result will only be used when \(h\) is small compared to \(x,\) say \(h\leq x^{0.99}\) (see Section \ref{proofthms}), in which case \((h/x)^{\varepsilon/2}\) delivers a good saving. 
 \begin{lemma}[Contribution from \(\Gamma_0^{\pm},\) part II]\label{gamma02} Let \(\varepsilon>0\) be small. Suppose \(h\) and \(y\leq x\) are in the range \eqref{yandh}, and \(x\) is sufficiently large in terms of \(\varepsilon.\) If \(x^{1/2-\varepsilon}<y\leq x^{1-2\varepsilon}\) then
 \[\int_{\Gamma\cap\Gamma_0^{\pm}}\min\left\{\frac{h^2}{x^2}|s-1|,\frac{h}{x}\right\}|\zeta(s,y)x^s|~|\mathrm{d}s|\ll_\varepsilon \left(\frac{h}{x}\right)^{1+\varepsilon/2}\Psi(x,y)(\log y)e^{-c_3u} \]
  for a small absolute constant \(c_3>0.\)
 \end{lemma}
 \begin{proof}
First suppose that \(x^{1/2-\varepsilon}<y\leq x^{1/2-\varepsilon/2}/2.\) 
 Using \[\min\left\{\frac{h^2}{x^2}|s-1|,\frac{h}{x}\right\}\leq \frac{h^2}{x^2}|s-1|,\] the calculation from the previous lemma gives a bound of \[\ll\frac{h^2}{x^2}\Psi(x,y)e^{-c_2u}\int_{1-\alpha}^{(1-\varepsilon)/2}\left(\frac{y^2}{x^{1-\varepsilon}}\right)^{\delta-(1-\alpha)}~\mathrm{d}\delta\ll\frac{h^2}{x^2}\frac{\Psi(x,y)e^{-c_2u}}{\log(x^{1-\varepsilon}/y^2)}\ll\frac{h^2}{x^2}\Psi(x,y)e^{-c_2u}.\]
Similarly, we achieve the same bound when   \( x^{1/2-\varepsilon/2}/2<y\leq 2x^{1/2-\varepsilon/2}\) by estimating this last integral using the triangle inequality. 

It remains to consider \(2x^{1/2-\varepsilon/2}<y\leq x^{1-2\varepsilon}.\) As in the previous lemma, \(|s-1|\ll |t|\) on \(\Gamma\cap \Gamma_0^+,\) so it suffices to bound 
\[\int_{1-\delta+it\in \Gamma\cap\Gamma_0^+}\min\left\{\frac{h^2}{x^2}|t|,\frac{h}{x}\right\}|\zeta(s,y)x^s|~|\mathrm{d}s|.\]
Let \(\delta_*=\max\{\delta: 1-\delta+it\in \Gamma\cap\Gamma_0^+ \textrm{ and } |t|\leq x/h\}.\)
 The portion of the integral with \(\delta\leq \delta_*\) is  
\begin{equation}\label{deltastar}
\ll_\varepsilon \frac{h^2}{x^2}\Psi(x,y)e^{-c_2u} \int_{1-\alpha}^{\delta_*}\left(\frac{y^2}{x^{1-\varepsilon}}\right)^{\delta-(1-\alpha)}~\mathrm{d}\delta \ll \frac{h^2}{x^2}\frac{\Psi(x,y)e^{-c_2u}}{\log(y^2/x^{1-\varepsilon})} \left(\frac{y^2}{x^{1-\varepsilon}}\right)^{\delta_*-(1-\alpha)}.
\end{equation}
Note that \[1+2\varepsilon\leq \frac{1}{1-2\varepsilon}\leq u<\frac{2}{1-\varepsilon}.\]
We have \[\frac{\log (u+1)}{\varepsilon^2(\log y)}y^{\delta_*-(1-\alpha)}\leq \frac{x}{h}.\] Raising both sides to the power of \(2-u(1-\varepsilon)>0,\) we have 
\[\left(\frac{\log (u+1)}{\varepsilon^2(\log y)}\right)^{2-u(1-\varepsilon)}\left(\frac{y^2}{x^{1-\varepsilon}}\right)^{\delta_*-(1-\alpha)}\leq \left(\frac{x}{h}\right)^{2-u(1-\varepsilon)}.\]
This leads to a bound for \eqref{deltastar} of \[\ll_\varepsilon\Psi(x,y)e^{-c_3u}\left(\frac{h}{x}\right)^{u(1-\varepsilon)}\frac{(\log y)^{2-u(1-\varepsilon)}}{\log(y^2/x^{1-\varepsilon})}\ll_\varepsilon \left(\frac{h}{x}\right)^{1+\varepsilon/2}\Psi(x,y)(\log y)e^{-c_3u}.\] 
Similarly, the portion of the integral with \(\delta>\delta_*\) is 
\begin{align*}
\ll\frac{h}{x}\Psi(x,y)e^{-c_2u}\int_{\delta_*}^{(1-\varepsilon)/2}\left(\frac{y}{x^{1-\varepsilon}}\right)^{\delta-(1-\alpha)}~\mathrm{d}\delta
&\ll\frac{h}{x}\frac{\Psi(x,y)e^{-c_2u}}{\log (x^{1-\varepsilon}/y)}\left(\frac{y}{x^{1-\varepsilon}}\right)^{\delta_*-(1-\alpha)}\\&\ll_\varepsilon\left(\frac{h}{x}\right)^{u(1-\varepsilon)}\Psi(x,y)e^{-c_3u}(\log y)^{-u(1-\varepsilon)}\\&\leq \left(\frac{h}{x}\right)^{1+\varepsilon/2}\Psi(x,y)e^{-c_3u}.
\end{align*}

 An identical calculation yields the same bound for \(\Gamma\cap\Gamma_0^-.\) 
 \end{proof}
\begin{lemma}[Contribution from the remainder] Let \(\varepsilon>0\) be small. Suppose \(y\leq x\) is in the range \eqref{yandh}, and \(x\) is sufficiently large in terms of \(\varepsilon.\) Then
\[\int_{\Gamma\setminus \left(\Gamma_0 \cup\left(\bigcup_\rho \Gamma_\rho\right)\right)}|\zeta(s,y)x^s|~|\mathrm{d}s|\ll \Psi(x,y)\sqrt{\log x \log y}  \frac{H}{x^{1/2-2\varepsilon}}.\]
\end{lemma}\begin{proof}
 Given that this integral lies on the line \(\Re(s)=(1+\varepsilon)/2,\) we have
 \begin{align*}
\int_{\Gamma\setminus (\Gamma_0\cup(\bigcup_\rho \Gamma_\rho))}|\zeta(s,y)x^s| ~|\mathrm{d}s|
 \ll H \frac{\zeta(\alpha,y)x^\alpha }{x^{(1-\varepsilon)(\alpha-1/2-\varepsilon/2)}} \ll \Psi(x,y)\sqrt{\log x \log y}  \frac{H}{x^{1/2-2\varepsilon}} 
,\end{align*}
 where we have used that \(1-\alpha\leq \varepsilon.\) \end{proof}
\begin{lemma}[Contribution from the horizontal]  Let \(\varepsilon>0\) be small. Suppose \(y\leq x\) is in the range \eqref{yandh}, and \(x\) is sufficiently large in terms of \(\varepsilon.\)  Then
 \[ \int_{\substack{\sigma \pm iH\\\Gamma(\pm H)\leq \sigma\leq \alpha}}|\zeta(s,y)x^s |~|\mathrm{d}s|\ll\frac{\Psi(x,y)}{\sqrt{u}}.\] 
 \end{lemma}
 \begin{proof}We have
 \begin{align*} 
\int_{\Gamma(H)+iH}^{\alpha+iH}  \left|\zeta(s,y)x^s\right| ~|\mathrm{d}s| \ll \zeta(\alpha,y)x^\alpha\int_{(1+\varepsilon)/2}^{\alpha} x^{-(1-\varepsilon)(\alpha-\sigma)}~\mathrm{d}\sigma
  \ll \frac{\zeta(\alpha,y)x^\alpha }{\log x.}\ll  \frac{\Psi(x,y)}{\sqrt{u}},
 \end{align*} and we can similarly  bound the lower horizontal.
\end{proof}

 \begin{lemma}[Contribution near real axis]\label{contribnearreal} Uniformly for \(2\leq y\leq x\) we have 
\[\frac{1}{2\pi i}\int_{\alpha-i/\log y}^{\alpha+i/\log y}\zeta(s,y)x^s ~\mathrm{d}s=\alpha\Psi(x,y)\left(1+O\left(\frac{1}{u}+\frac{\log y}{y}\right)\right),\] and the same estimate is true for
\[\frac{1}{2\pi}\int_{\alpha-i/\log y}^{\alpha+i/\log y}|\zeta(s,y)x^s| ~|\mathrm{d}s|.\] 
We also have \[\frac{1}{2\pi}\int_{\alpha-i/\log y}^{\alpha+i/\log y}|s-1||\zeta(s,y)x^s| ~|\mathrm{d}s|\ll\frac{\Psi(x,y)\log(u+1)}{\log y}.\] 
\end{lemma}
\begin{proof} 
The first two results follow from \cite[Lemma 12]{hilten} and then \cite[Thm.~1]{hilten} (noting that both sides in \cite[Thm.~1]{hilten} are positive so we are free to divide by \(1+O(1/u+\log y/y)\)). 

The final estimate follows from the fact that  \(|s-1|\leq |\Re(s-1)|+|\Im(s-1)|\leq |1-\alpha|+1/\log y\ll\log(u+1)/\log y,\) by the approximate formula \eqref{alpha}.
\end{proof}
\begin{lemma}[Contribution at intermediate height]\label{intheight}  Let \(\varepsilon> 0.\)  Suppose \(\log x< y\leq x\) is sufficiently large  in terms of \(\varepsilon.\) Then uniformly \[\int_{\substack{\alpha+it\\1/\log y\leq |t|\leq \log(u+1)/(\varepsilon^2\log y)}}|\zeta(s,y)x^s| ~|\mathrm{d}s|\ll_{\varepsilon} \Psi(x,y)e^{-c_4u/(\log(u+1))^2} \] and  \[\int_{\substack{\alpha+it\\1/\log y\leq |t|\leq \log(u+1)/(\varepsilon^2\log y)}}|s-1||\zeta(s,y)x^s| ~|\mathrm{d}s|\ll_{\varepsilon} \frac{\Psi(x,y)e^{-c_4u/(\log(u+1))^2}}{\log y} \] for some absolute constant \(c_4>0.\)    
\end{lemma}
\begin{proof} We  use  \eqref{quantbounds} which implies that if \(1/\log y\leq|t|\leq 2/\varepsilon^2\) then
\[|\zeta(\alpha+it,y)|\ll\zeta(\alpha,y)e^{-c_0u/ (\log(u+1))^2}\] for some \(c_0>0.\)  Indeed,  we have \(|1-\alpha|\ll \log(u+1)/\log y\) from the approximation \eqref{alpha}, so that \(t^2/(t^2+(1-\alpha)^2)\gg\min\{1, t^2/(1-\alpha)^2\}\gg 1/(\log(u+1))^2.\)  The first integral is then certainly
\[ \ll \frac{\log (u+1)}{\varepsilon^2\log y}\frac{\Psi(x,y)\sqrt{\log x \log y}}{e^{c_0 u/(\log (u+1))^2}},\] which gives  the first desired estimate.  

Given that \(|s-1|\leq |\Re(s-1)|+|\Im(s-1)|\ll_\varepsilon \log(u+1)/\log y,\) again because \(|1-\alpha|\ll \log(u+1)/\log y,\) the second estimate follows.
\end{proof}
Assuming  \(H\leq x^{13/30-5\varepsilon},\)  we have shown for \(y\leq x^\varepsilon\) satisfying  \eqref{yandh}, with \(x\)  sufficiently large in terms of \(\varepsilon,\) that \(\Psi(x+h,y)-\Psi(x,y)\) equals \begin{equation}\label{total}\begin{split}\frac{h}{2\pi i}\frac{1}{x}\int_{\Gamma\cap\Gamma_0}\zeta(s,y)x^s ~\mathrm{d}s &+O_\varepsilon\left(\frac{h^2}{x^2}\frac{\Psi(x,y)\log(u+1)}{\log y}\right)\\&+O_\varepsilon\left(\frac{h}{x}\Psi(x,y)e^{-c((\log x)/\log \log x)^{1/3}}\right)+O\left(\frac{\Psi(x,y)}{H\sqrt{u}}\right)\\&+O\left(\frac{h}{x}\Psi(x,y)\sqrt{\log x \log y}  \frac{H}{x^{1/2-2\varepsilon}} \right)+O\left(\frac{x\log x}{H}\right).\end{split}\end{equation}  We set \(H= xe^{(\log x)^{9/10}}/h.\) By \eqref{yandh} we have \(h\geq x^\theta,\) and so \(H\leq x^{1-\theta+o(1)}\). The earlier bound on \(H\) is then satisfied if, say,  \(\theta\geq 17/30+10\varepsilon.\) By \eqref{density}, the final error term in \eqref{total} is certainly \(\ll (h/x)\Psi(x,y)/e^{\sqrt{\log x}}.\)  

 Recall that in the definition of \(\Gamma\) we have \(|t|\leq H,\) and so the integral in \eqref{total} may depend on \(h,\) which gives us a main term not of the desired form.  To overcome this, we first define \(\Gamma_\infty\) to be the contour \(\Gamma\) but without the restriction that  \(|t|\leq H,\) and set \[f_\theta(x,y)=\frac{1}{2\pi i}\frac{1}{x}\int_{\Gamma_\infty\cap\Gamma_0}\zeta(s,y)x^s ~\mathrm{d}s.\] Note that this function depends on \(\theta\) because the contour depends on \(\varepsilon,\) which in turn depends on \(\theta.\) The next lemma allows us to approximate the main term of \eqref{total} by \(f_\theta(x,y)h,\) whilst introducing an additional error of only  \[\ll_\varepsilon\frac{h}{x}\frac{\Psi(x,y)}{e^{(\varepsilon/2)(\log x)^{9/10}}},\] completing the proof of Theorem \ref{mainfunc} assuming \(y\leq x^\varepsilon\). 
\begin{lemma}[Contribution of \(\Gamma_0^{\pm}\) at large height]\label{largeheightmain} Let \(\varepsilon>0\) be small.  Suppose \(y\leq x^{1-2\varepsilon}\) is in the range \eqref{yandh}, and \(x\) is sufficiently large in terms of \(\varepsilon.\)  Then
 \[ \int_{(\Gamma_\infty\cap\Gamma_0^\pm)\setminus (\Gamma\cap\Gamma_0^\pm)}|\zeta(s,y)x^s|~|\mathrm{d}s |\ll_\varepsilon\frac{\Psi(x,y)}{e^{{(\varepsilon/2)(\log x)}^{9/10}}}.\] 
\end{lemma}
\begin{proof} For \(1-\delta+it\in\Gamma_\infty\cap \Gamma_0^+,\)  we have \[|t|=\frac{ \log (u+1)}{\varepsilon^2(\log y)y^{1-\alpha}}y^\delta.\]  
Let \(\delta^*=\max\{\delta: 1-\delta+it\in \Gamma_\infty\cap\Gamma_0^+ \textrm{ and } |t|\leq H\}.\) If \(|t|<H\) at \(\delta=\delta^*,\) then we are done. We may therefore assume this is not the case, so that  \(y^{\delta^*-1+\alpha}\gg_\varepsilon H (\log y)/\log(u+1)\gg e^{(\log x)^{9/10}/2}.\)
By a similar calculation to that in Lemma \ref{gamma0}, we have   
\begin{align*}\int_{(\Gamma_\infty\cap\Gamma_0^+)\setminus (\Gamma\cap\Gamma_0^+)}|\zeta(s,y)x^s|~|\mathrm{d}s|  &\ll\int_{1-\delta+it\in(\Gamma_\infty\cap\Gamma_0^+)\setminus (\Gamma\cap\Gamma_0^+)}\left|\frac{\zeta(\alpha+it,y)x^\alpha}{x^{(1-\varepsilon)(\alpha-(1-\delta))}}\right|\left|\frac{\mathrm{d}}{\mathrm{d}\delta}\Gamma_0^+(\delta)\right|~\mathrm{d}\delta\\ &\ll_\varepsilon \zeta(\alpha,y)x^\alpha e^{-c_0u}\log(u+1)\int_{\delta^*}^{(1-\varepsilon)/2}\left(\frac{y}{x^{1-\varepsilon}}\right)^{\delta-(1-\alpha)}~\mathrm{d}\delta
\\ &\ll_\varepsilon \Psi(x,y)e^{-c_2u}\left(\frac{y}{x^{1-\varepsilon}}\right)^{\delta^*-1+\alpha}.
\end{align*}
Since \[\left(\frac{y}{x^{1-\varepsilon}}\right)^{\delta^*-1+\alpha}\leq \left(\frac{1}{x^{\varepsilon}}\right)^{\delta^*-1+\alpha}\leq \left(\frac{1}{y^{\varepsilon}}\right)^{\delta^*-1+\alpha}\ll_\varepsilon e^{-(\varepsilon/2)(\log x)^{9/10}},\] we obtain the required bound. 

The proof for \((\Gamma_\infty\cap\Gamma_0^-)\setminus (\Gamma\cap\Gamma_0^-)\) follows from an identical calculation.
  \end{proof}
 
 The condition \(y\leq x^\varepsilon\) emerges from the contribution from the zeros \(\rho.\)
 This requirement ultimately comes from needing a large local zero-free region in Lemma \ref{shifting} due to the slow decay of the error term for the relevant Perron integral. To overcome this barrier, in the next section we employ a smooth weight which leads to a better decay rate than that offered by the Perron integral. We are then left to deal with large \(y,\) due to the conditions imposed when bounding the error from \(\Gamma_0^\pm,\) which we do in a later section by different means. 
  
 \section{Smoothing to extend the range}\label{smoothing}
  In this section we consider the case when 
  \[x^\theta \leq h \leq x \quad \textrm{ and } \quad
  x^\varepsilon\leq y\leq x^{1-2\varepsilon},\] 
 where \(17/30<\theta\leq 1,\) with \(\varepsilon=\varepsilon(\theta)>0\) suitably small, and  \(x\)  sufficiently large in terms of \(\varepsilon.\)
 We show that Theorem \ref{mainfunc} holds in this range if we further assume \(y\leq x^{1/2-\varepsilon}.\) Without this assumption, the error term we obtain in this section is in general weaker than desired, and so requires additional work which we handle later on.
 
For  \(0<\kappa<1\) small, which we specify later, let \(W=W_\kappa:(0,\infty)\rightarrow [0,1]\) be a smooth function which equals \(1\) on \((0,1]\) and equals \(0\) on \((1+\kappa,\infty).\) We may choose \(W\) so that the derivative bound \begin{equation}\label{derivbound}\int_1^{1+\kappa}|W^{(m)}(x)| ~\mathrm{d} x\ll_m \kappa^{1-m}\end{equation} holds for all \(m\geq 0.\) Consider the weighted count of \(y(1+\kappa)\)-smooth numbers
 \[\Psi_W(x,y)=\sum_{\substack{n={p_1}^{a_1}\dots {p_k}^{a_k} \\ n\leq x}}W\left(\frac{p_1}{y}\right)^{a_1}\dots W\left(\frac{p_k}{y}\right)^{a_k}.\] \begin{lemma}[Weighted count approximates unweighted count]\label{weightedcount}  Let \(17/30<\theta\leq 1,\) and let \(0<\varepsilon<1.\) The estimate 
 \[\Psi_W(x+h,y)-\Psi_W(x,y)=\Psi(x+h,y)-\Psi(x,y) +O_{\varepsilon}\left(h\frac{\Psi(x,y)}{x}\left(\kappa  +\frac{y}{x\log x}\right)\right)\] holds uniformly  for   \[
x^\theta \leq h \leq x \quad \textrm{ and } \quad
  x^\varepsilon\leq y\leq x \] provided  \(\kappa y\geq 1.\)
 \end{lemma}
 \begin{proof}
 The difference between the weighted count and \(\Psi(x+h,y)-\Psi(x,y)\) is  
 \begin{equation}\label{difference}\sum_{\substack{n={p_1}^{a_1}\dots {p_k}^{a_k} \\ x<n\leq x+h\\ \exists p_i>y}}W\left(\frac{p_1}{y}\right)^{a_1}\dots W\left(\frac{p_k}{y}\right)^{a_k}.\end{equation}
 If we first suppose that \(y\leq h,\) then  the difference \eqref{difference} is at most
 \[\sum_{y< p\leq y(1+\kappa) \ }\sum_{x/p<m\leq (x+h)/p}1\leq  \sum_{y< p\leq y(1+\kappa)  } \left(\frac{h}{p}+1\right)\ll\kappa h+\kappa y\ll \kappa h.\]
 If instead we suppose that \(h<y,\) then the difference \eqref{difference} is bounded above by
  \[ \sum_{x/(y(1+\kappa))\leq m\leq (x+h)/y \ } \sum_{x/m<p\leq (x+h)/m} 1 \ll \sum_{x/(y(1+\kappa))\leq m\leq (x+h)/y \ }\frac{h}{m\log (h/m)} .\] Here we used the simple sieve bound, a simple version of the Brun-Titchmarsh theorem, valid for any \(x'\geq 0\) and any \(h'\geq 2\) (see e.g.~\cite[Cor.~3.4]{montvaug}),  \begin{equation}\label{sieve}
  \sum_{x'<p\leq x'+h'}1\ll\frac{h'}{\log h'}.\end{equation}
 We have \(m\asymp x/y.\) Thus \(h/m\asymp hy/x>h^2/x\geq x^{2\theta-1}\geq x^{2/15}.\) The above sum is then
 \[\ll\frac{hy}{x\log x}\sum_{x/(y(1+\kappa))\leq m\leq (x+h)/y \ }1.\]
Since \(h<y,\)  the range of the sum has size at most \((x+y)/y-x(1-\kappa)/y+1=\kappa x/y+2.\) Thus, our sum is 
 \[\ll \frac{hy}{x\log x}\left(\frac{\kappa x}{y}+1\right)=\frac{h}{\log x}\left(\kappa  +\frac{y}{x}\right).\]
  Given that \(y\geq x^\varepsilon,\) and so \(\Psi(x,y)\gg_\varepsilon x\) (see e.g.~ \cite{hild2}),  the proof is complete. \end{proof}
 
 We estimate the weighted sum as we did for the unweighted sum, replacing \(\zeta(s,y)\) with \[\zeta_W(s,y)=\prod_p\left(1-\frac{W(p/y)}{p^s}\right)^{-1}.\] 
 To do this, we need the following result, which is a weighted version of Lemma \ref{shifting}. 

\begin{lemma}
[Shifting into local zero-free region yields cancellation, smoothed version]\label{smoothedshifting} Let \(z\geq 2.\)  Let \(\varepsilon>0\) be small, and let  \(0< \eta\leq1/2.\)   Write \(s_0=\sigma_0+it_0 \neq 1,\) where \(1-(1-\varepsilon)\eta\leq \sigma_0<1+1/\log z.\) Suppose there are no zeros of \(\zeta(s)\) in the region 
\[\Re(s)>1-\eta \quad \textrm{and} \quad \Im(s)\in (t_0-10z^{3\varepsilon\eta},t_0+10z^{3\varepsilon\eta}).\] If 
\( z^{-\varepsilon\eta}\leq \kappa<1,\) then for some \(b=b(\varepsilon)<1\) we have 
\begin{align*}
\left|\sum_{n \geq 1}\frac{\Lambda(n)}{n^{s_0}}W\left(\frac{n}{z}\right)\right|= O&\left(\frac{z^{1-\sigma_0}+1}{|t_0|}\right)+O_\varepsilon\left((\log (|t_0|+2))^b\right)+O_\varepsilon\left(\frac{\log z}{z^{\eta}}\right) \\
&+O_\varepsilon\left(\frac{((\log(|t_0|+z^\eta+2))^b+1/\eta)\log(1/(\kappa\eta))}{z^{\varepsilon\eta/2}}\right).
\end{align*}  
\end{lemma}
\begin{proof}
For \(s=\sigma+it\in \mathbb{C},\) the Mellin transform is given by \[\widetilde{W}(s)=\int_0^\infty W(x)x^{s-1}~\mathrm{d} x.\]
Suppose \(\sigma\geq 1/10,\)  say. The integral is  absolutely convergent here, so we may integrate by parts \(m\geq 1\) times to obtain
\begin{equation}\label{intparts0}\widetilde{W}(s)=(-1)^m\int_1^{1+\kappa} W^{(m)}(x)\frac{x^{s+m-1}}{s(s+1)\dots(s+m-1)}~\mathrm{d} x. \end{equation} 
This integral is absolutely convergent for all \(s\in \mathbb{C}\setminus\{0,-1,\dots,-m+1\},\) so \eqref{intparts0} defines a meromorphic continuation of \(\widetilde{W}(s).\) Moreover, using \eqref{derivbound} we have \begin{equation}\label{intparts}\widetilde{W}(s)\ll \frac{1+(1+\kappa)^{\sigma+m-1}}{|t|^m}\int_1^{1+\kappa}|W^{(m)}(x)|~\mathrm{d}x\ll_m \frac{1}{\kappa ^{m-1} |t|^m}\end{equation}
 whenever  \(-\infty<\sigma\leq 10\)  (say, so that the bound does not depend on \(\sigma\))
and \(t\neq 0.\)
By the Mellin inversion formula, for any \(\sigma>0,\) we have \[W(x)=\frac{1}{2\pi i}\int_{\sigma-i\infty}^{\sigma+i\infty}\widetilde{W}(s)x^{-s}~ \mathrm{d}s.\]
Then, with \(\sigma=1-\sigma_0+1/\log z,\) and \(m\geq 2,\) we have \begin{align*}\sum_{n\geq 1} \frac{\Lambda(n)}{n^{s_0}}W\left(\frac{n}{z}\right) &=\frac{1}{2\pi i}\int_{\sigma-i\infty}^{\sigma+i\infty}\frac{-\zeta'}{\zeta}(s+s_0)\widetilde{W}(s)z^{s}~ \mathrm{d}s \\ &=\frac{1}{2\pi i}\int_{\sigma-iT}^{\sigma+iT}\frac{-\zeta'}{\zeta}(s+s_0)\widetilde{W}(s)z^{s}~ \mathrm{d}s+O_m\left(\int_{T}^\infty \frac{(\log z)z^{1-\sigma_0}}{\kappa^{m-1}t^{m}} ~\mathrm{d}t \right).\end{align*}
A change of variables in the main term, and estimating the error term, gives
\[\frac{1}{2 \pi i}\int_{1+1/\log z+it_0-iT}^{1+1/\log z+it_0+iT}\frac{-\zeta'}{\zeta}(s)\widetilde{W}(s-s_0)z^{s-s_0}~ \mathrm{d}s+O_m\left(\frac{(\log z)z^{1-\sigma_0}}{(\kappa T)^{m-1}}\right).\]  Choose \(m\in[1+1/\varepsilon,2+1/\varepsilon)\) and  \(T\geq \kappa^{-1} z^{2\eta/(m-1)},\)  so the error term is  \(O_\varepsilon((\log z)/z^\eta).\) 
 
 We shift the contour left to \(\Re(s)=r=1-(1-\varepsilon/2)\eta.\) As in the proof of Lemma \ref{shifting}, we separately consider the cases where \(|t_0|\geq 3\kappa^{-1}z^{2\eta/(m-1)}\) (set \(T=2\kappa^{-1} z^{2\eta/(m-1)}\)) and \(|t_0|< 3\kappa^{-1}z^{2\eta/(m-1)}\) (set \(T=4\kappa^{-1} z^{2\eta/(m-1)}\)). In either case, 
 \( T\leq 4z^{\varepsilon\eta+2\eta/(m-1)}\leq 4   z^{3\varepsilon\eta}.\)   The above integral becomes  
 \begin{align*}\frac{-\zeta'}{\zeta}&(s_0)+z^{1-s_0}\widetilde{W}(1-s_0)1_{|t_0|<3\kappa^{-1}z^{2\eta/(m-1)}} \\
 +&\frac{1}{2\pi i}\left(\int_{r+it_0-iT}^{r+it_0+iT}+\int_{1+1/\log z+it_0-iT}^{r+it_0-iT}+\int_{r+it_0+iT}^{1+1/\log z+it_0+iT}\right)\frac{-\zeta'}{\zeta}(s)\widetilde{W}(s-s_0)z^{s-s_0}~\textrm{d}s.\end{align*}
 By Lemma \ref{improved}, we have \(|-\zeta'(s_0)/\zeta(s_0)|\leq 1/|t_0|+O_\varepsilon((\log(|t_0|+2))^b).\) From \eqref{intparts}, this time with \(m=1,\) we also have \(|z^{1-s_0}\widetilde{W}(1-s_0)|\ll z^{1-\sigma_0}/|t_0|.\) 
 
  We can bound  the upper horizontal integral by
\[\ll \max_{\sigma\in[r,1+1/\log z]}\left|\frac{-\zeta'}{\zeta}(\sigma+it_0+ iT)\right|\frac{z^{1-\sigma_0}}{\kappa^{m-1}T^m}\ll_\varepsilon \frac{(\log(|t_0|+z^{3\varepsilon\eta}+2))^b}{z^\eta}.\]  The lower horizontal integral is bounded similarly. The vertical integral is 
  \begin{align*} &\ll z^{r-\sigma_0} \max_{t\in[-T,T]}\left|\frac{-\zeta'}{\zeta}(r+it_0+ it)\right|\int_{-T}^T\left|\widetilde{W}(r-\sigma_0+it)\right| ~\mathrm{d}t\\
&\ll_\varepsilon\frac{(\log(|t_0|+T+2))^b+ 1/\eta }{z^{\varepsilon\eta/2}}\left(\int_{1/\kappa}^{T} \frac{1}{\kappa^{m-1}t^m}~\textrm{d}t+\int_\eta^{1/\kappa}\frac{1}{t}~\textrm{d}t+\int_{-\eta}^{\eta} \frac{1}{\eta}~\textrm{d}t\right)\\
&\ll \frac{(\log(|t_0|+z^\eta+2))^b+ 1/\eta }{z^{\varepsilon\eta/2} }\left(1+\log (1/(\kappa\eta))+1\right),
\end{align*}
where we have used Lemma \ref{improved}, the fact that \(|r-\sigma_0+it|\geq \max\{r-\sigma_0,|t|\} \geq \max\{\varepsilon\eta/2,|t|\},\) and the bound \eqref{intparts}.\end{proof}
 
 Let \(\kappa=y^{-\varepsilon\lambda}\leq x^{-\varepsilon^2\lambda}.\) We  replace \(y^\delta\) with \(y^{3\varepsilon\delta}\) in the definition \(\Gamma_\rho,\) meaning
 \[\widetilde{\Gamma_\rho'}(t)=1-(1-\varepsilon)\nu+i(\gamma+10t), \quad \quad t\in[-y^{3\varepsilon\nu},y^{3\varepsilon\nu}] ,\] and  \[\widetilde{\Gamma_\rho^\pm}(\delta)=1-(1-\varepsilon)\delta+i(\gamma\pm 10y^{3\varepsilon\delta}), \quad \quad \delta\in\left[\nu, \frac{1}{2}\right] \] produces a new path \(\widetilde{\Gamma_\rho}=\widetilde{\Gamma_\rho'}\cup\widetilde{\Gamma_\rho^+}\cup\widetilde{\Gamma_\rho^-},\) from which we obtain a new contour \(\widetilde{\Gamma}\) instead of \(\Gamma.\) 
 We use Lemma \ref{smoothedshifting} instead of Lemma \ref{shifting} to deduce a saving on \(\widetilde{\Gamma}.\)
 We  then proceed the same way as the unsmoothed count with virtually the same estimates, though with \(\zeta_W(\alpha,y)\) in place of \(\zeta(\alpha,y)\)  (note that the proof of the version of Lemma \ref{contzeros} we need carries through since \(|\widetilde{\Gamma_\rho}|\ll y^{3\varepsilon/2}\leq x^{3\varepsilon/2}\)). This is permissible for error terms because we claim that \(|\zeta_W(\alpha+it,y)|\ll |\zeta(\alpha+it,y)|.\) This follows from the bound
\[\log \left|\frac{\zeta_W(\alpha+it,y)}{\zeta(\alpha+it,y)}\right|\leq-\sum_{y<p\leq y(1+\kappa)}\log\left(1-\Re\frac{W(p/y)}{p^{\alpha+it}}\right)\ll \sum_{y<p\leq y(1+\kappa)}\frac{1}{p^\alpha}\ll \kappa y^{1-\alpha}\leq 1,\] where the last inequality uses \(1-\alpha\leq \varepsilon\lambda\) by \eqref{alphalambda}. 
  We then let \[f_\theta(x,y)=\frac{1}{2\pi i}\frac{1}{x}\int_{\Gamma_\infty\cap \Gamma_0}\zeta_W(s,y)x^s ~\mathrm{d}s,\] and as before we are able to conclude  the proof of the statement in Theorem \ref{mainfunc}, provided \(y\leq x^{1/2-\varepsilon}\) (where this restriction comes from the bound over \(\Gamma_0^\pm\) in Lemma \ref{gamma0}). 
  
  If, however, we have \(x^{1/2-\varepsilon}<y\leq x^{1-2\varepsilon},\) then we use (a suitable modification of) Lemma \ref{gamma02} in place of Lemma \ref{gamma0}, leading to an additional error term of size \[\ll_\varepsilon\left(\frac{h}{x}\right)^{1+\varepsilon/2}\Psi(x,y)(\log y)e^{-c_3u}.\] 
  When \(h\leq x^{0.99},\) say,  this error is of no significance. To address when this is not the case, we proceed via an entirely different approach. 
 
    \section{Large \(y\) via counting primes}\label{largey}
  Let us handle the scenario when 
 \[x/y^{13/30-\delta}\leq h\leq x  \quad \mathrm{ and } \quad  x^{\delta'}\leq y\leq 2x\]  
    via a more direct approach, where \(\delta>0\) is small and \(0<\delta'<1\),  for \(x\)  sufficiently large.   In particular, setting \(\delta'=0.49\) and \(\delta=\varepsilon\) allows us to prove Theorem \ref{mainfunc} when  \[x^{17/30+10\varepsilon}\leq h\leq x  \quad \mathrm{ and } \quad  x^{1-2\varepsilon}\leq y\leq 2x,\] or when
      \[ x^{0.99}\leq h\leq x  \quad \mathrm{ and } \quad x^{1/2-\varepsilon}\leq y\leq x^{1-2\varepsilon},\]
which we were unable to deal with by the saddle point method.    
  
 Our initial manoeuvre is similar to that of Hildebrand \cite[\S 5]{hild2}. Writing \(p_1,\dots, p_k\) for pairwise distinct prime variables, from the inclusion-exclusion principle we have
 \begin{align}\label{primesum}
\Psi(x+h,y)-\Psi(x,y)&=\sum_{x< n\leq x+h}1+\sum_{1\leq k\leq u+1}(-1)^k\sum_{ \substack{p_1,\dots, p_k>y,~m\geq 1: \\x<p_1\dots p_k m\leq x+h}}1 \nonumber 
\\&=h+O(1) + \sum_{1\leq k\leq u+1}(-1)^k\sum_{p_1,\dots, p_{k-1}>y,~m\geq 1}\sum_{\substack{p_k>y:\\x<p_1\dots p_km\leq x+h}}1 
 \end{align}
 The restriction on \(k\) is present because the first inner sum is empty when  \(y^k>2x.\)  By introducing a logarithmic weight, we estimate the final inner sum  with a small error term.
\begin{lemma}[Primes in short intervals] \label{primesshortlemma} Let \(\delta>0\) be small. There exists a small constant \(c=c(\delta)>0\) such that for all \(x',h'\geq 10\) satisfying \((x')^{17/30+\delta}\leq h'\leq x',\) we have 
\[\sum_{x'<p\leq x'+h'}1=\frac{h'}{\log x'}\left(1+O_\delta\left(e^{-c(\log x')^{1/3}/(\log \log x')^{1/3}}\right)+O\left(\frac{h'}{x'\log x'}\right)\right).\]
\end{lemma}
\begin{proof}
 As \(2\leq h'\leq x'<p\leq x'+h',\) we have \[1<\frac{\log p}{\log x'}\leq \frac{\log (x'+h')}{\log x'}\leq 1+ O\left(\frac{h'}{x'\log x'}\right).\] The  simple sieve bound \eqref{sieve} then yields
  \[\sum_{x'<p\leq x'+h'}1=\frac{1}{\log x'}\sum_{x'<p\leq x'+h'}\log p +O\left(\frac{(h')^2}{x'(\log x')^2}\right).\] 
  We also have\begin{equation}\label{logp}\begin{split}
  0\leq \sum_{x'<n\leq x'+h'}\Lambda(n)-\sum_{x'<p\leq x'+h'}\log p &\leq \sum_{2\leq k\leq \log (2x')/\log 2}~\sum_{\substack{p:\\ x'<p^k\leq x'+h'}}\log (2x')\\&\ll (\log x')^2\sqrt{x'}.\end{split}\end{equation}
  
  We have the explicit formula (see e.g~\cite[Thm.~2.5]{montvaug}) for \(2\leq T\leq x'\) given by
  \[\sum_{n\leq x'}\Lambda(n)=x'-\sum_{|\Im (\rho)| \leq T}\frac{(x')^\rho}{\rho}+O\left(\frac{x'(\log x')^2}{T}\right),\] where the sum is over zeros \(\rho\) of \(\zeta(s)\) with \(\Re(\rho)\geq 0.\) 
  Then by Taylor's theorem (see \eqref{taylor0}) 
  \[\sum_{x'<n\leq x'+h'}\Lambda(n)=h'+O\left(h'\sum_{|\Im (\rho)| \leq T}{(x')^{\Re(\rho)-1}}\right)+O\left(\frac{x'(\log x')^2}{T}\right).\]
Choose \((x'/h')(x')^{\delta/2} \leq T\leq (x')^{13/30-\delta/2}.\) This is always possible provided \(h'\geq (x')^{17/30+\delta}.\)    For a suitably small constant \(c=c(\delta)>0,\) we estimate the sum, as we did in \eqref{sumzeros}, by \[\sum_{x'<n\leq x'+h'}\Lambda(n)=h'+O_\delta\left(h'e^{-c(\log x')^{1/3}/(\log \log x')^{1/3}}\right)+O\left(h'\frac{(\log x')^2}{(x')^{\delta/2}}\right),\]  and the latter error term may be absorbed into the former, as can the error term \eqref{logp}. Combining these estimates delivers the desired conclusion. 
\end{proof}

Now consider the innermost sum in \eqref{primesum}, which is of the form   
\begin{equation}\label{primesum2} \sum_{\max\{y,x/r\}<p\leq (x+h)/r}1,\end{equation} with \(r\) a positive integer.
  Writing \(x'=x/r\) and \(h'=h/r,\) we see that
  \(h'\geq (x')^{17/30+\delta}\) if and only if \(h\geq x^{17/30+\delta} r^{13/30-\delta}.\)
  If \(r\leq x/y,\) then we indeed have  
  \[x^{17/30+\delta}r^{13/30-\delta}\leq x^{17/30+\delta}(x/y)^{13/30-\delta}= x/y^{13/30-\delta}\leq h.\]
    Thus, we may apply our estimate  for primes in short intervals in Lemma \ref{primesshortlemma} to the triple sum in \eqref{primesum} when \(p_1\dots p_{k-1}m\leq x/y,\) leading to a main term from the triple sum of
\[h\sum_{1\leq k\leq u}(-1)^k\sum_{\substack{p_1,\dots,p_{k-1}>y,~ m\geq 1:\\ p_1\dots p_{k-1}m\leq x/y}} \frac{1}{p_1\dots p_{k-1}m\log (x/(p_1\dots p_{k-1}m))}\] (we may ignore when \(u<k\leq u+1\) as the inner sum would be empty), and an error of 
\[\ll_\delta\sum_{1\leq k\leq u+1}~\sum_{\substack{p_1,\dots,p_{k-1}>y,~ m\geq 1:\\ p_1\dots p_{k-1}m\leq x/y}}\frac{h}{p_1\dots p_{k-1}m\log y}\left(e^{-c(\log y)^{1/3}/(\log \log y)^{1/3}}+\frac{h}{x\log y}\right).\] 
Each positive integer \(r\leq x/y\) can have at most \(O_{\delta'}(1)\) prime factors greater than \(y.\) Thus \(r\) can be represented in the form \(p_1\dots p_{k-1}m,\) such that  \(p_1,\dots ,p_{k-1}>y,\) in at most \(O_{\delta'}(1)\) ways. It follows that \begin{equation}\label{rbound}\sum_{\substack{p_1,\dots,p_{k-1}>y,~ m\geq 1:\\ p_1\dots p_{k-1}m\leq x/y}}\frac{1}{p_1\dots p_{k-1}m}\ll_{\delta'} \sum_{r\leq x/y}\frac{1}{r}\ll \log x.\end{equation} The above error term is therefore  
\[\ll_{\delta'} h \left(e^{-c(\log y)^{1/3}/(\log \log y)^{1/3}}+\frac{h}{x\log y}\right).\]
    
 The contribution to the sum in \eqref{primesum2}  from the remaining range \(r>x/y\) is 
 \[\sum_{y<p\leq (x+h)/r}1\leq \sum_{x/r<p\leq (x+h)/r}1\ll \frac{h}{r\log x}< \frac{y}{x}\frac{h}{\log x}\]
 by the sieve bound \eqref{sieve}. 
Using this, along with similar reasoning to that establishing \eqref{rbound}, the contribution to the triple sum in \eqref{primesum} when \(y<p_1\dots p_{k-1}m\leq (x+h)/y\)  is
 \[ \ll\sum_{1\leq k\leq u+1}\sum_{\substack{p_1,\dots,p_{k-1}>y,~ m\geq 1:\\x/y< p_1\dots p_{k-1}m\leq (x+h)/y}} \sum_{y<p\leq (x+h)/(p_1\dots p_{k-1}m)}1\ll_{\delta'}\left(\frac{h}{y}+1\right) \frac{y}{x}\frac{h}{\log x}=\frac{h\left(h+y\right)}{x\log x}.\]

Writing \[g(x,y)=1+\sum_{1\leq k\leq u}(-1)^k\sum_{\substack{p_1,\dots,p_{k-1}>y,~ m\geq 1:\\ p_1\dots p_{k-1}m\leq x/y}} \frac{1}{p_1\dots p_{k-1}m\log (x/(p_1\dots p_{k-1}m))},\]  we therefore obtain
 \[\Psi(x+h,y)-\Psi(x,y)=h\left(g(x,y)+O_{\delta',\delta}\left(e^{-c\delta'(\log x)^{1/3}/(\log \log x)^{1/3}}+\frac{h+y}{x\log x}\right)\right),\] 
  which is of the shape claimed in Theorem \ref{mainfunc}  once we observe (see e.g.~ \cite{hild2}) that \(\Psi(x,y)\gg_{\delta'} x\) in our range of \(y.\) 
  
 \section{Proof of Theorems 1.1 and 1.2}\label{proofthms}
Let \(\varepsilon>0\) be small, and assume \(\theta\geq 17/30+10\varepsilon.\)  Suppose \(h\) and \(y\) satisfy \eqref{yandh}, with \(x\) sufficiently large in terms of \(\varepsilon.\) If \(y\leq x^{1/2-\varepsilon},\) then the statement in Theorem \ref{mainfunc} holds by the work of Sections \ref{contour} and \ref{smoothing}. On the other hand, when  \(y\geq x^{1-2\varepsilon},\) the statement in Theorem \ref{mainfunc} holds by the work of Section \ref{largey}. 
 
  Let us consider the  range \(x^{1/2-\varepsilon}< y< x^{1-2\varepsilon}\) (and so \(\Psi(x,y)\gg x\)).
The results of Sections \ref{contour} and \ref{smoothing} imply, in particular, that if \(x^{\theta}\leq h\leq x^{0.99}\) then
\[\frac{\Psi(x+h,y)-\Psi(x,y)}{h}=f_\theta(x,y)+O_\varepsilon\left(e^{-c((\log x)/\log \log x)^{1/3}}\right).\]  If instead  \(h\geq x^{0.99},\)  then from Section \ref{largey}  we have   \[\frac{\Psi(x+h,y)-\Psi(x,y)}{h}=g(x,y)+O_\varepsilon\left(e^{-c((\log x)/\log \log x)^{1/3}}+\frac{h+y}{x\log x}\right).\]
On setting \(h= x^{0.99},\) we see that \[f_\theta(x,y)=g(x,y)+O_\varepsilon\left(e^{-c((\log x)/\log \log x)^{1/3}}\right).\] This establishes the desired result in the range \(x^{1/2-\varepsilon}< y< x^{1-2\varepsilon},\) and hence  concludes the proof of Theorem \ref{mainfunc}.

 To deduce Theorem \ref{main}, we set \(h=x\) in Theorem \ref{mainfunc}, so that 
\[\frac{\Psi(2x,y)-\Psi(x,y)}{x}=f_\theta(x,y)+O\left(\frac{\Psi(x,y)}{x}\frac{\log (u+1)}{\log y}\right)\] for all \(y\) satisfying \eqref{yandh}. 
We now appeal to the short interval result of Hildebrand \cite{hild2} \[\frac{\Psi(2x,y)-\Psi(x,y)}{x}=\frac{\Psi(x,y)}{x}+O\left(\frac{\Psi(x,y)}{x}\frac{\log(u+1)}{\log y}\right)\] for \(e^{(\log \log x)^{5/3+1/100}}\leq y\leq x.\) This estimate  also holds if \(y>x\) by the prime number theorem, as we are simply estimating the density of composite numbers in the interval \((x,2x].\)  This concludes the proof of Theorem \ref{main}.
  \section{Assuming the Riemann Hypothesis}\label{riemannsection}
  In this section we assume the Riemann Hypothesis. \subsection{Saddle point method} Suppose  \(\theta\geq 1/2+10\varepsilon\) for a suitably small \(\varepsilon>0.\) Let us further suppose that \((\log x)^K\leq y\leq x^{1-2\varepsilon},\) with \(K\geq 10/\varepsilon^2,\) and that \(x\) sufficiently large in terms of \(\varepsilon.\) Observe that on this range \(\alpha\geq 1-1/K+o(1)>(1+\varepsilon)/2.\) It is also known by \eqref{densityinu}  that  \(\Psi(x,(\log x)^K)=x^{1-1/K+o(1)}.\) 
  
 Recall from Perron's formula \eqref{perron} that if \(y,h,H\in [2,x]\) then
\[\Psi(x+h,y)-\Psi(x,y)=\frac{1}{2\pi i}\int_{\alpha-iH}^{\alpha+iH}\zeta(s,y)\frac{(x+h)^s-x^s}{s} ~\mathrm{d}s+O\left(\frac{x\log x}{H}\right).\] We set \(H=x^{1+2/(\varepsilon
K)}/h,\) so that the above error term becomes \(\ll h(\log x)/x^{2/(\varepsilon
K)}\ll h\Psi(x,y)/x^{1+1/(\varepsilon
K)}.\)

We then perform a shift of the integral to the left as before. 
We may ignore the contribution from the \(\Gamma_\rho\) since these all already lie on the vertical `remainder' line \(\Re (s)= (1+\varepsilon)/2.\) Thus we have
\[\frac{1}{2\pi i}\left(\int_{\Gamma\cap\Gamma_0}+\int_{\Gamma\setminus\Gamma_0}+\int_{\substack{\sigma \pm iH\\ \Gamma(\pm H)\leq \sigma\leq \alpha}}\right)\zeta(s,y)\frac{(x+h)^s-x^s}{s} ~\mathrm{d}s.\]
We set \(\eta=1/2\) and apply a version of Lemma \ref{powersave} for our range of \(y.\) We see from the proof that this is possible as \(y^{\varepsilon\eta}=y^{\varepsilon/2} \)  beats any power of \(\log y.\) The `remainder' and horizontal integrals once again contribute
\[\ll\frac{h}{x}\Psi(x,y)\sqrt{\log x\log y}\frac{H}{x^{1/2-2\varepsilon}}+\frac{\Psi(x,y)}{H}.\] By assumption we have \(h\geq x^\theta,\)  and so \(H\leq x^{1+2/(\varepsilon K)-\theta}\). As \(\theta\geq 1/2+10\varepsilon,\) we have \(H\leq x^{1/2+2/(\varepsilon K)-10\varepsilon}\leq x^{1/2-8\varepsilon}.\)  Thus the above error contribution is \[\ll\frac{h}{x}\Psi(x,y)\sqrt{\log x\log y}\frac{1}{x^{6\varepsilon}}+\frac{h}{x}\Psi(x,y)\frac{1}{x^{2/(\varepsilon K)}}\ll\frac{h}{x}\Psi(x,y)\frac{1}{x^{1/K}}.\]
 
As before, we are left with  \begin{equation}\label{rheqn}\frac{1}{2\pi i}\frac{h}{x}\int_{\Gamma\cap \Gamma_0}\zeta(s,y)x^s~\mathrm{d}s +O\left(\int_{\Gamma\cap\Gamma_0}\min\left\{\frac{h^2}{x^2}|s-1|,\frac{h}{x}\right\}|\zeta(s,y)x^s|~|\mathrm{d}s|\right).\end{equation} Adapting the proof of Lemma \ref{largeheightmain},  we may replace the main term here with \(f_\theta(x,y)h,\) where 
\[f_\theta(x,y)=\frac{1}{2\pi i}\frac{1}{x}\int_{\Gamma_\infty\cap\Gamma_0}\zeta(s,y)x^s ~\mathrm{d}s,\] introducing  an error of 
\[\ll_\varepsilon\frac{h}{x}\frac{\Psi(x,y)}{x^{2\varepsilon/(2\varepsilon K)}}=\frac{h}{x}\Psi(x,y)\frac{1}{x^{1/K}}.\]
By Lemmas \ref{contribnearreal} and \ref{intheight}, along with suitable adaptations of Lemmas \ref{gamma0} and \ref{gamma02},  the error term in \eqref{rheqn} is \[\ll_\varepsilon\frac{h^2}{x^2}\Psi(x,y)\frac{\log(u+1)}{\log y}+\left(\frac{h}{x}\right)^{1+\varepsilon/2}\Psi(x,y)(\log y)e^{-c_3u}1_{x^{1/2-\varepsilon}< y\leq x^{1-2\varepsilon}}.\] 

Thus, when \((\log x)^K\leq  y\leq x^{1/2-\varepsilon},\) we have \[\Psi(x+h,y)-\Psi(x,y)=f_\theta(x,y) h+O_\varepsilon\left(\frac{h}{x}\Psi(x,y)\left(\frac{1}{x^{1/K}}+\frac{h\log(u+1)}{x\log y}\right)\right).\]  
Our use of smoothing plays no part here, as this was used only for the paths \(\Gamma_\rho,\) which we have no need to consider by assuming the Riemann Hypothesis. We postpone addressing the case \(x^{1/2-\varepsilon}<y\leq x^{1-2\varepsilon}\) until we have established additional results. 

 Our strategy in Section \ref{proofthms} for approximating \(f_\theta(x,y)\) by \(\Psi(x,y)/x\) relied on considering when \(h=x\) and applying the short interval result of Hildebrand. This approach remains valid when \(e^{(\log \log x)^{5/3+1/100}}\leq y\leq x^{1/2-\varepsilon},\)  leading to 
\begin{equation}\label{error1}\Psi(x+h,y)-\Psi(x,y)=\frac{h\Psi(x,y)}{x}\left(1+O_\varepsilon\left(\frac{\log(u+1)}{\log y}\right)\right).\end{equation} The absence of a factor of \(\alpha,\) as compared to the statement in Theorem \ref{mainriemann}, is of no concern, because \(\alpha=1-O(\log(u+1)/\log y)\)  by the approximate formula \eqref{alpha}.

We require a different approach in order to cover smaller \(y.\) 
From Lemmas \ref{gamma0} (after suitable modification for our range of \(y\)), \ref{contribnearreal}, and \ref{intheight}, we have
 \begin{align*}\frac{1}{2\pi i}\int_{\Gamma\cap \Gamma_0}\zeta(s,y)x^s~\mathrm{d}s =\alpha\Psi(x,y)&+O\left(\frac{\Psi(x,y)}{u}\right)+O\left(\frac{\Psi(x,y)\log y}{y}\right)\\&+O_\varepsilon\left(\frac{\Psi(x,y)}{e^{c_4u/(\log (u+1))^2}}\right)+O_\varepsilon\left(\frac{\Psi(x,y)}{e^{c_2u}}\right).\end{align*}
 When \((\log x)^K\leq  y\leq x^{1/2-\varepsilon},\)  all these error terms are \(\ll_\varepsilon \Psi(x,y)/u,\) and so
 \begin{equation}\label{error2}\Psi(x+h,y)-\Psi(x,y)=\alpha\frac{h\Psi(x,y)}{x}\left(1+O_\varepsilon\left(\frac{1}{u}+\frac{1}{x^{1/K}}+\frac{h\log(u+1)}{x\log y}\right)\right).\end{equation}
We combine \eqref{error1} and \eqref{error2} by taking the minimum of their error terms.

Now suppose that \((\log x)^{2+10\varepsilon}\leq y\leq (\log x)^{K}.\) Then \(\alpha\geq 1-1/(2+10\varepsilon)+o(1)>1/2+\varepsilon.\)  
At this level of smoothness, the small size of \(\Psi(x,y)\) presents an obstacle to obtaining an adequate  error term from Perron's formula. We therefore use a version of Perron's formula appearing in \cite[\S 3]{harp3} (see also \cite{fouvten}) which is sensitive to the size of the set of smooth numbers rather than the whole interval.
This gives \(\Psi(x+h,y)-\Psi(x,y)\) equal to  \[\frac{1}{2\pi i}\int_{\alpha-iH}^{\alpha+iH}\zeta(s,y)\frac{(x+h)^s-x^s}{s} ~\mathrm{d}s+O\left(\Psi(x,y)\left(\frac{\sqrt{\log x \log y}}{\sqrt{H}}+\frac{1}{H^{1/50}}\right)\right).\] 

As stated earlier, inspecting the proof of Lemma \ref{powersave} with \(\eta=1/2\) reveals that this result also holds our range of \(y,\) given that \(y^{\varepsilon\eta}=y^{\varepsilon/2}\) beats any power of  \(\log y.\) We shift the integral left, encountering  error terms on the `remainder' line (here we use the fact that \(\alpha\geq 1/2+\varepsilon\)) and the horizontal paths  of size
\[\ll \frac{h}{x}\Psi(x,y)\sqrt{\log x\log y}\frac{H}{x^{9\varepsilon/20}}+\frac{\Psi(x,y)}{H}.\] Let \(H=x^{\varepsilon/4},\) so that our errors so far are \(\ll \Psi(x,y)/x^{\varepsilon/200}\leq (h/x)\Psi(x,y)/x^{\varepsilon/400}\leq (h/x)\Psi(x,y)/x^{1/K},\)  provided that we take \(\theta\geq 1-\varepsilon/400\) so that \(h/x\geq x^{\theta}/x\geq 1/x^{\varepsilon/400}.\) 
The main term is handled as before.

\subsection{Counting primes} Now suppose \(y\geq x^{1/100}\)   and \(x/y^{1/2-\delta}\leq h\leq x,\) for small \(\delta>0\) and for \(x\) sufficiently large.  (As in Section \ref{largey}, the following argument works even when replacing \(1/100\) with an  arbitrarily small  \(\delta'>0.\)) In particular, with \(\delta=\varepsilon,\) this covers the case when \(x^{1-2\varepsilon}\leq y\leq 2x\) and \(x^{1/2+10\varepsilon}\leq h\leq x.\)

Consider our direct approach counting primes in short intervals. As \(\Re(\rho)=1/2,\) and the number of non-trivial zeros with imaginary height bounded in magnitude by \(T\) is \(\ll T\log T,\) we have  \[\sum_{x'<n\leq x'+h'}\Lambda(n)=h'+O\left(\frac{h'}{\sqrt{x'}}T\log T\right)+O\left(\frac{x'(\log x')^2}{T}\right).\] Choose \((x'/h')(x')^{\delta/2} \leq T\leq (x')^{1/2-\delta/2},\) which is always possible provided \(h'\geq (x')^{1/2+\delta}.\)   If so, both error terms become \(\ll h'(\log x')^2/(x')^{\delta/2}.\) As in Section \ref{largey}, this leads to  \[\sum_{x'<p\leq x'+h'}1=\frac{h'}{\log x'}\left(1+O\left(\frac{(\log x')^2 }{(x')^{\delta/2}}\right)+O\left(\frac{h'}{x'\log x'}\right)\right).\]

Now consider the triple sum in \eqref{primesum}. With \(x'=x/r\) and \(h'=h/r,\)  we see that
  \(h'\geq (x')^{1/2+\delta}\) if and only if \(h\geq x^{1/2+\delta} r^{1/2-\delta}.\) If \(r\leq x/y,\) we indeed have  
  \[x^{1/2+\delta}r^{1/2-\delta}\leq x^{1/2+\delta}(x/y)^{1/2-\delta}= x/y^{1/2-\delta}\leq h.\] Then the triple sum in \eqref{primesum} on the range \(p_1\dots p_{k-1} m\leq x/y\) has a main term of \[h\sum_{1\leq k\leq u}(-1)^k\sum_{\substack{p_1,\dots,p_{k-1}>y,~ m\geq 1:\\ p_1\dots p_{k-1}m\leq x/y}} \frac{1}{p_1\dots p_{k-1}m\log (x/(p_1\dots p_{k-1}m))}\] with error terms of 
\[\ll\sum_{1\leq k\leq u+1}~\sum_{\substack{p_1,\dots,p_{k-1}>y,~ m\geq 1:\\ p_1\dots p_{k-1}m\leq x/y}}\left(\frac{h\log (x/p_1\dots p_{k-1}m)}{x^{\delta/2}(p_1\dots p_{k-1}m)^{1-\delta/2}}+\frac{h^2}{xp_1\dots p_{k-1}m(\log y)^2}\right).\]   
	By the reasoning of \eqref{rbound}, and comparison with an integral, the first error term is \[\ll\frac{h\log x}{x^{\delta/2}}\sum_{1\leq k\leq u+1}\sum_{r\leq x/y}\frac{1}{r^{1-\delta/2}}\ll\frac{h\log x}{\delta x^{\delta/2}}\left(\frac{x}{y}\right)^{\delta/2}\ll\frac{h\log x}{\delta y^{\delta/2}}.\] Similarly, the second error term is \(\ll h^2/(x\log x).\) The range \(p_1\dots p_{k-1} m >x/y\) in \eqref{primesum} is dealt with exactly as before.  We arrive at \[\Psi(x+h,y)-\Psi(x,y)=h\left(g(x,y)+O\left(\frac{\log x}{\delta y^{\delta/2}}\right)+O\left(\frac{h+y}{x\log x}\right)\right),\] where \[g(x,y)=1+\sum_{1\leq k\leq u}(-1)^k\sum_{\substack{p_1,\dots,p_{k-1}>y,~ m\geq 1:\\ p_1\dots p_{k-1}m\leq x/y}} \frac{1}{p_1\dots p_{k-1}m\log (x/(p_1\dots p_{k-1}m))}.\] Once again, \(\Psi(x,y)\gg x\) on this range of \(y.\) 
	
	Finally, we cover the remaining range \(x^{1/2-\varepsilon}< y< x^{1-2\varepsilon}\) by adapting the approach of Section \ref{proofthms}. This completes the proof of Theorem \ref{mainriemann}.
\bibliography{main}

\providecommand{\bysame}{\leavevmode\hbox to3em{\hrulefill}\thinspace}
\providecommand{\MR}{\relax\ifhmode\unskip\space\fi MR }
\providecommand{\MRhref}[2]{%
  \href{http://www.ams.org/mathscinet-getitem?mr=#1}{#2}
}
\providecommand{\href}[2]{#2}
\begin{thebibliography}{Har12b}

\bibitem[BT05]{bretten}
R.~de~la Bret\`eche and G.~Tenenbaum, \emph{Propri\'{e}t\'{e}s statistiques des
  entiers friables}, Ramanujan J. \textbf{9} (2005), no.~1-2, 139--202.

\bibitem[CEP83]{canerdpom}
E.~R. Canfield, P.~Erd\H{o}s, and C.~Pomerance, \emph{On a problem of
  {O}ppenheim concerning ``factorisatio numerorum''}, J. Number Theory
  \textbf{17} (1983), no.~1, 1--28.

\bibitem[Dav80]{dav}
H.~Davenport, \emph{Multiplicative number theory}, second ed., Graduate Texts
  in Mathematics, vol.~74, Springer-Verlag, New York-Berlin, 1980, Revised by
  Hugh L. Montgomery.

\bibitem[FT91]{fouvten}
\'{E}. Fouvry and G.~Tenenbaum, \emph{Entiers sans grand facteur premier en
  progressions arithmetiques}, Proc. London Math. Soc. (3) \textbf{63} (1991),
  no.~3, 449--494.
  
\bibitem[FG93]{friedgran}
J.~B. Friedlander and A.~Granville, \emph{Smoothing ``smooth'' numbers},
  Philos. Trans. Roy. Soc. London Ser. A \textbf{345} (1993), no.~1676,
  339--347.

\bibitem[Gra93]{granv2}
A.~Granville, \emph{Integers, without large prime factors, in arithmetic
  progressions. {II}}, Philos. Trans. Roy. Soc. London Ser. A \textbf{345}
  (1993), no.~1676, 349--362.

\bibitem[Gra08]{granv}
\bysame, \emph{Smooth numbers: computational number theory and beyond},
  Algorithmic number theory: lattices, number fields, curves and cryptography,
  Math. Sci. Res. Inst. Publ., vol.~44, Cambridge Univ. Press, Cambridge, 2008,
  pp.~267--323.
  
\bibitem[GM24]{guthmayn}
L.~Guth and J.~Maynard, \emph{{New large value estimates for Dirichlet
  polynomials}}, preprint, arXiv:2405.20552, 2024.  

\bibitem[Har12a]{harp3}
A.~J. Harper, \emph{{Bombieri--Vinogradov and Barban--Davenport--Halberstam
  type theorems for smooth numbers}}, preprint, arXiv:1208.5992, 2012.

\bibitem[Har12b]{harp2}
\bysame, \emph{On a paper of {K}. {S}oundararajan on smooth numbers in
  arithmetic progressions}, J. Number Theory \textbf{132} (2012), no.~1,
  182--199.

\bibitem[HB88]{heath}
D.~R. Heath-Brown, \emph{The number of primes in a short interval}, J. Reine
  Angew. Math. \textbf{389} (1988), 22--63.

\bibitem[Hil86]{hild2}
A.~Hildebrand, \emph{On the number of positive integers {$\leq x$} and free of
  prime factors {$>y$}}, J. Number Theory \textbf{22} (1986), no.~3, 289--307.

\bibitem[HT86]{hilten}
A.~Hildebrand and G.~Tenenbaum, \emph{On integers free of large prime factors},
  Trans. Amer. Math. Soc. \textbf{296} (1986), no.~1, 265--290.

\bibitem[HT93]{hilten2}
\bysame, \emph{Integers without large prime factors}, J. Th\'{e}or. Nombres
  Bordeaux \textbf{5} (1993), no.~2, 411--484.

\bibitem[Hux72]{hux}
M.~N. Huxley, \emph{On the difference between consecutive primes}, Invent.
  Math. \textbf{15} (1972), 164--170.

\bibitem[Mat16]{matom}
K.~Matom\"{a}ki, \emph{Another note on smooth numbers in short intervals}, Int.
  J. Number Theory \textbf{12} (2016), no.~2, 323--340.

\bibitem[MR16]{matrad}
K.~Matom\"{a}ki and M.~Radziwi\l\l, \emph{Multiplicative functions in short
  intervals}, Ann. of Math. (2) \textbf{183} (2016), no.~3, 1015--1056.

\bibitem[MV07]{montvaug}
H.~L. Montgomery and R.~C. Vaughan, \emph{Multiplicative number theory. {I}.
  {C}lassical theory}, Cambridge Studies in Advanced Mathematics, vol.~97,
  Cambridge University Press, Cambridge, 2007.

\bibitem[Ram76]{ram}
K.~Ramachandra, \emph{Some problems of analytic number theory}, Acta Arith.
  \textbf{31} (1976), no.~4, 313--324.

\bibitem[RS62]{rossch}
J.~B. Rosser and L.~Schoenfeld, \emph{Approximate formulas for some functions
  of prime numbers}, Illinois J. Math. \textbf{6} (1962), 64--94.

\bibitem[Sou08]{sound}
K.~Soundararajan, \emph{The distribution of smooth numbers in arithmetic
  progressions}, Anatomy of integers, CRM Proc. Lecture Notes, vol.~46, Amer.
  Math. Soc., Providence, RI, 2008, pp.~115--128.

\bibitem[Sou10]{sound2}
\bysame, \emph{{Smooth numbers in short intervals}}, preprint, arXiv:1009.1591,
  2010.

\bibitem[Tit86]{titc}
E.~C. Titchmarsh, \emph{The theory of the {R}iemann zeta-function}, second ed.,
  The Clarendon Press, Oxford University Press, New York, 1986, Edited and with
  a preface by D. R. Heath-Brown.

\bibitem[Xua99]{xuan}
T.~Z. Xuan, \emph{On smooth integers in short intervals under the {R}iemann
  hypothesis}, Acta Arith. \textbf{88} (1999), no.~4, 327--332.

\bibitem[You]{you}
K.~Younis, \emph{{Smooth numbers with prescribed digits}}, in preparation.

\end{thebibliography}
\bibliographystyle{amsalpha}

\end{document}